\documentclass[12pt,reqno]{amsart}
\usepackage{graphicx}
\usepackage{cases}
\usepackage{amscd}
\usepackage{amsfonts}
\usepackage{amsmath}
\usepackage{graphicx}
\usepackage{subfigure}
\usepackage{amsthm}

\usepackage{fullpage}
\usepackage{fancyhdr,graphicx,amsmath,amssymb}
\usepackage[ruled,vlined]{algorithm2e}
\include{pythonlisting}
\allowdisplaybreaks

\usepackage{mathtools}
\usepackage{cite}

\usepackage{hyperref}
\hypersetup{colorlinks,allcolors=blue}

\theoremstyle{remark}
\newtheorem{example}{\textbf{Example}}[section]
\numberwithin{equation}{section}

\usepackage{color}
\usepackage{datetime}

\def\tr{\textcolor{red}}
\def\tb{\textcolor{blue}}

\allowdisplaybreaks
\makeatletter
\newcommand\figcaption{\def\@captype{figure}\caption}
\newcommand\tabcaption{\def\@captype{table}\caption}
\makeatother

\usepackage{geometry}
\usepackage{multirow}

\def\bq{\begin{equation}}
\def\eq{\end{equation}}
\def\br{\begin{eqnarray}}
\def\er{\end{eqnarray}}
\def\brr{\bq\begin{array}{rlll}}
\def\err{\end{array}\eq}
\def\R{\mathbb{R}}
\def\text#1{\hbox{#1}}
\newtheorem{thm}{Theorem}[section]
\newtheorem{lem}{Lemma}[section]

\newtheorem{rem}{Remark}[section]
\newtheorem{prop}[thm]{Proposition}

\newcommand{\bsub}{\begin{subequations}}
\newcommand{\esub}{\end{subequations}$\!$}

\title[A dynamic mass transport method for PNP equations]{A dynamic mass transport method for  Poisson-Nernst-Planck equations }
\author[  ]{Hailiang Liu and Wumaier Maimaitiyiming}
\address{$^\dagger$ Department of Mathematics, Iowa State University, USA}
\email{hliu@iastate.edu}
\address{Department of Mathematics, University of California Los Angeles, USA} 
\email{wumaier@math.ucla.edu}
\keywords{PNP equations, Optimal transport, Wasserstein distance, Positivity, Energy dissipation}
\subjclass{35Q92, 65N08, 65N12}
\begin{document}


\begin{abstract} A dynamic mass-transport method  
is proposed for approximately solving the Poisson–Nernst–Planck (PNP) equations. The semi-discrete scheme based on the JKO type variational formulation naturally enforces solution positivity and the energy law as for the continuous PNP system. The fully discrete scheme is further formulated as a constrained minimization problem, shown to be solvable, and satisfy all three solution properties (mass conservation, positivity and energy dissipation) independent of time step size or the spatial mesh size. Numerical experiments are conducted to validate convergence of the computed solutions and verify the structure preserving property of the proposed scheme.
\end{abstract}
\maketitle

\section{Introduction}

In this paper, we consider a  time-dependent system of Poisson-Nernst-Planck (PNP)  equations. Such system has been widely used to describe charge transport in diverse applications such as biological membrane channels \cite{TT53, CE93, Ei98},  electrochemical systems \cite{BTA04},
 and semiconductor devices \cite{PRS90,S84}. 

PNP equations consist of Nernst--Planck (NP) equations that describe the drift and diffusion of ion species, and the Poisson equation that describes the electrostatic interaction. Such mean field approximation of diffusive ions  admits several variants, and in non-dimensional form we consider the following
\begin{subequations}\label{PNP1}
\begin{align}
\label{PNP}
&  \partial_t \rho_i = \nabla\cdot \left [  D_i(x) \left( \nabla \rho_i + z_i\rho_i \nabla \phi \right) \right], \quad x\in \Omega\subset \R^d, \quad t>0, \\
\label{Ps}
&  -\nabla \cdot(\epsilon(x) \nabla \phi)  =  f(x)+\sum_{i=1}^{s}z_i\rho_i,
\end{align}
\end{subequations}
subject to initial data $\rho_i(x, 0)=\rho_i^{in}(x)\geq 0$ ($i=1, \cdots, s$) and appropriate boundary conditions to be specified in section \ref{sec2}.  The equations are valid in a bounded domain $\Omega$ with boundary $\partial \Omega$ and for time $t\geq 0$. Here $\rho_i=\rho_i(x,t)$ is the charge carrier density for the $i$-th species,  and $\phi=\phi(x,t)$ the electrostatic potential. $D_i(x)$ is the diffusion coefficient, $z_i$ is the rescaled charge. In the Poisson equation, $\epsilon(x)$ is the permittivity, $f(x)$ is the permanent (fixed) charge density  of  the system, $s$ is the number of species. 

Due to the wide variety of devices modeled by the PNP equations, computer simulation for this system of differential equations is of great interest. However, the PNP system is a strongly coupled system of nonlinear equations, also, the PNP system as a gradient flow can take very long time evolution to reach steady states. Hence, designing efficient and stable numerical methods for the PNP system remains an active area of research (see, e.g., \cite{LMJSC, LWWYZ20, LWYY21, SX21, SZL21, DWZ22}).

PNP system possesses two immediate properties: it preserves the non-negativity of $\rho_i$  and conserves total mass. Therefore, we can consider non-negative initial data with mass one, so that the density is in the set of probability measures $\mathcal{P}(\Omega)$ on $\Omega$. The third property is the dissipation of the total energy, which can be expressed as follows. Given energy 
\begin{equation}\label{Intro-Energy}
E= \int_{\Omega}\bigg( \sum_{i=1}^s \rho_i\log \rho_i+\frac{1}{2}(f+\sum_{i=1}^s z_i\rho_i)\phi \bigg )d x +B,
\end{equation}
with boundary correction term $B$, the NP equation (\ref{PNP}) can be written as the gradient flow
$$
\partial_t \rho_i= \nabla \cdot(\rho_i D_i(x)\nabla (\delta_{\rho_i} E)).
$$
Differentiating the energy along solutions of the PNP system, one formally obtains the energy dissipation along the gradient flow 
$$
\frac{dE}{dt} = -\int _{\Omega} \sum_{i=1}^s D_i(x)\rho_i|\nabla (\log \rho_i+z_i\phi) |^2 dx \leq 0,
$$
which indicates that the solution evolves in the direction of steepest descent of the energy. This property entails a characterization of the set of stationary states, and provides a useful tool to study its stability. Numerical methods for (\ref{PNP1}) are desired to attain all three properties at the discrete level, which are rather challenging. 
\subsection{Related work}
The most common numerical approach is the direct discretization of (\ref{PNP1}) using classical finite difference, finite volume, finite element, or discontinuous Galerkin methods \cite{FMEKLL13, LW14,HP16, MXL16, FKL17, LW17, GH17, LMCICP, LMJSC, HPY19,DWZ19,HH19,LWWYZ20, LWYY21, SX21, SZL21, DWZ22}. Such methods are explicit or semi-implicit in time, so the per time computation is cheap. But it is often challenging to ensure both unconditional positivity and discrete energy decay simultaneously. The nonlinearity also complicates the way to obtain solutions when applying implicit or semi-implicit solvers.

For the gradient flow
$$
\partial_t \rho= -\nabla^{W}_\rho E(\rho),
$$
with the gradient with respect to the quadratic Wasserstein metric $W_2(\cdot,\cdot)$, the minimizing movement approximation (see \cite{Savar05} and the references therein) 
\begin{equation}\label{GJKO}
  \rho^{n+1}=\text{argmin}_{\rho\in K} \left\{\frac{1}{2\tau}W^2_2(\rho^n,\rho)+E(\rho)\right\}, \quad \rho^0=\rho^{in}(x), 
\end{equation}
also named Jordan-Kinderlehrer-Otto (JKO) scheme (Jordan et al. \cite{JKO98}), defines a sequence $\{\rho^{n}\}$ in the probability space $K$ to approximate the solution $\rho(x, n\tau)$, where $\tau>0$ is the time step. Since $\rho^n$ is in the probability space, thus method (\ref{GJKO}) is positivity and mass preserving. The fact $W_2(\rho^n,\rho)\geq 0$  ensures the energy dissipation for any $\tau>0$. We refer the interested reader to \cite{LWC19, Katy19,MP19} for some JKO-type schemes for gradient flows in the probability space. 

The authors in \cite{KMX17} constructed the JKO scheme for a two species PNP system with constant coefficient $D_i(x)=1$ and $\epsilon(x)=1$. The existence of the unique minimizer to the JKO scheme and convergence of the minimizer to the weak solution of the PNP system have been established in \cite{KMX17}. However, the fully discrete JKO-type scheme for the multi-species variable coefficient PNP system on a bounded domain has not been studied yet. This is what we aim to accomplish in this paper.

\subsection{Our contributions}
 In the present work, we will extend Benamou-Brenier’s dynamic formulation \cite{BB00} for the Wasserstein distance to the present setting that resolve the aforementioned issues. 
\begin{itemize} 
\item In a  general setting with mixed types of boundary conditions, we identify the total energy functional which is dissipating along the solution trajectories, we also establish (Theorem \ref{lb}) lower energy bounds with coercivity for such a functional.  
\item We construct a Wasserstein-type distance and  formulate a corresponding variational scheme. The update at each step reduces to solving a constrained minimization problem, for which we prove unique solvability (Theorem \ref{thm2.3}). Three solution properties: mass conservation, positivity, and energy dissipation are shown to be preserved in time (Theorem \ref{thm2.4}).  
\item We further convert the variational scheme into a dynamic formulation, which for variable diffusion coefficients extends the classical Bennamou-Breiner formulation. To reduce computational cost, we use a local approximation for the artificial time in the constraint transport equation by a one step difference and the integral in time by a one term quadrature.  The resulting minimization problem is shown to be a first order time consistent scheme for the PNP system (Theorem \ref{thm2.5}). 
\item We present a fully discrete scheme, and prove its unique solvability (Theorem \ref{thm3.1}).
\item We establish that for any fixed time step and spatial meth size, density positivity will be propagating over all time steps (Theorem \ref{thm3.2}). This is in sharp contrast to the work \cite{LWC19}, where Fisher information regularization is added to enforce solution positivity for the aggregation equation.  

\item The fully-discrete minimization problem reduces to a convex optimization problem with linear constraints, and can be solved by some efficient optimization solvers. Our numerical tests are conducted with a simple projected gradient algorithm. 
\item Numerical results are provided to demonstrate the superior performance of the proposed method. 
\end{itemize} 

\subsection{Organization} 
The paper is organized as follows. In the next section, we provide necessary background on the dynamical formulation of the PNP system, main solution properties and its relation with Wasserstein gradient flows. We then derive the semi-discrete scheme.  In Section 3, we introduce a fully discrete scheme and study the properties of this scheme. Numerical algorithms are given in Section 4. Numerical results are provided in Section 5, and the paper is concluded in Section 6.

\noindent{\bf Notation.}  We use $[n]$ to denote $\{1, 2, \cdots, n\}$ for any integer $n$. For vector $\phi$ (fully discrete case), its amplitude is denoted by $|\phi|$. For function $\phi$, $\|\phi\|$ is its $L^2$ norm.  

 \section{Model background and semi-discretization} \label{sec2} 
In this section we briefly review the model setup and the corresponding Wasserstein gradient flow. 

\subsection{Boundary conditions} \label{sec2.1} 
 Boundary conditions are a critical component of the PNP model and determine important qualitative behaviors of the solution. Let $\Omega$ be a bounded domain with Lipschitz boundary $\partial \Omega$.  
We use the no-flux boundary condition for the NP equations, i.e., 
\begin{equation}
\label{BNP}
   D_i(x) \left( \nabla \rho_i +z_i\rho_i \nabla \phi \right)\cdot \mathbf{n}=0,\quad x\in \partial\Omega, \quad i=1, \cdots, s.
  \end{equation}
Here, $\mathbf{n}$ is the outer unit normal at the boundary point $x\in \partial\Omega.$ 

The external electrostatic potential $\phi$ is influenced by applied potential, which can be modeled by prescribing a boundary condition.
Here we consider a general form of boundary conditions:  
\begin{equation}\label{BPO}
\alpha \phi+\beta \epsilon(x)\frac{\partial \phi}{\partial \mathbf{n}}=\phi^b, \quad x\in \partial \Omega.
\end{equation}
Here $\alpha, \beta$ are physical parameters such that $\alpha\cdot \beta \geq 0$, and $\phi^b=\phi^b(x, t)$ is a given function. With such setup, we are to solve the following initial-boundary value problem: 
\begin{equation}\label{PNP1F}
\left \{
\begin{array}{lll}
&  \partial_t \rho_i  = \nabla\cdot \left [  D_i(x) \left( \nabla \rho_i + z_i\rho_i \nabla \phi \right) \right ], &  x\in \Omega, \quad  t>0, \ i=1, \cdots, s,  \\
&  -\nabla \cdot(\epsilon(x) \nabla \phi)  = f(x)+\sum_{i=1}^{s}z_i\rho_i, &  x\in \Omega, \quad t>0,\\
& \rho_i(x, 0)=\rho_i^{in}(x)\geq 0, & x\in \Omega, \quad i=1, \cdots, s, \\
 &  \left( D_i(x) \left( \nabla \rho_i +z_i\rho_i \nabla \phi \right) \right)\cdot \mathbf{n}=0, & x\in \partial\Omega, \quad t>0, \ i=1, \cdots, s,  \\
 &\alpha \phi+\beta\epsilon(x) \frac{\partial \phi}{\partial \mathbf{n}}=\phi^b, & x\in \partial \Omega, t>0.  
 \end{array}
 \right.
\end{equation}
\begin{rem} (\ref{BPO}) includes three typical forms: 
(i) the Robin boundary condition ($\alpha=1, \beta>0$) models a capacitor \cite{FMEKLL13}, (ii) the Dirichlet boundary condition ($\alpha=1,\ \beta=0$) models an applied voltage,  and (iii) the Neumann boundary condition ($\alpha= 0, \ \beta=1$) models surface changes. The case of pure Neumann boundary conditions requires      the compatibility condition
\begin{equation}\label{CC}
\int_{\Omega} \left(  f(x)+\sum_{i=1}^{s}z_i\rho^{in}_i  \right) dx+\int_{\partial \Omega} \epsilon(x) \phi^b ds=0,
\end{equation}
and an additional constraint such as $\int_{\Omega}\phi(x, t)dx=0$ so that $\phi$ is uniquely defined.

\end{rem}

Any combination of these three types can be applied to $\phi$ on a disjoint partition of the boundary. In what follows, we set  
$$
\partial \Omega=\Gamma_D\cup \Gamma_N\cup \Gamma_R,
$$
and on each part, one type of boundary condition is imposed, i.e., 
$$
\alpha=\begin{dcases}
       1 , & \text{ on } \Gamma_D, \\
       0 , & \text{ on } \Gamma_N, \\ 
       1 , & \text{ on } \Gamma_R, \\
    \end{dcases} \quad 
    \beta=\begin{dcases}
       0 , & \text{ on } \Gamma_D, \\
       1 , & \text{ on } \Gamma_N, \\ 
       \beta_R , & \text{ on } \Gamma_R, \\
    \end{dcases}\quad 
        \phi^b=\begin{dcases}
       \phi^b_D , & \text{ on } \Gamma_D, \\
       \phi^b_N , & \text{ on } \Gamma_N, \\ 
       \phi^b_R , & \text{ on } \Gamma_R. \\
    \end{dcases}
$$
The existence and uniqueness of the solution for the nonlinear PNP boundary value problems with different boundary conditions have been studied in \cite{JL12, LW05, PJ97} for the 1D case and in \cite{BSW12, JJ85} for multi-dimensions.
\subsection{Energy functional: dissipation and coercivity}
 In the presence of homogeneous boundary conditions on $\phi$, i.e., $\phi^b=0$, the PNP system is energetically closed in the sense that the free energy functional associated to (\ref{PNP1}) is of form
\begin{equation}\label{energy1}
 E_0=\int_{\Omega}\bigg( \sum_{i=1}^s \rho_i\log \rho_i+\frac{1}{2}(f+\sum_{i=1}^s z_i\rho_i)\phi \bigg )d x, 
\end{equation}
which along solution trajectories is dissipating in time.  For general boundary conditions with $\phi^b \not=0$, we need to modify the energy so that it is still dissipating along the solution of the PNP system.
To this end, we differentiate (\ref{energy1}) along the solution of (\ref{PNP1}),  with integration by parts using (\ref{BNP}), we have  
\begin{align*}\label{energydiss1}
\frac{d}{dt}E_0(\rho, \phi)(t) & =
-\int _{\Omega} \sum_{i=1}^s D_i(x)\rho_i|\nabla (\log \rho_i+z_i\phi) |^2 dx  +\frac{1}{2}\int_{\partial\Omega} \epsilon(x) \left[ 
\phi (\partial_n \phi)_t - (\partial_n \phi )\phi_t \right] ds.
\end{align*}
Assume that $\phi^b$ does not depend on time, then $ \alpha \phi_t +\beta \epsilon(x) \partial_n \phi_t=0$ on $\partial \Omega$, this allows us to express the last term as 
$$
\frac{1}{2}\frac{d}{dt} \left[\int_{\Gamma_D} \epsilon(x)\phi^b_D\partial_n\phi ds- \int_{\Gamma_N}\phi^b_N\phi ds-\frac{1}{\beta_R}\int_{\Gamma_R}\phi^b_R\phi ds \right]. 
$$
Thus the modified total energy functional can be taken as    
\begin{equation}\label{energyNew}
\begin{aligned}
   E =E_0 -\frac{1}{2} \left[\int_{\Gamma_D} \epsilon(x)\phi^b_D\partial_n\phi ds- \int_{\Gamma_N}\phi^b_N\phi ds-\frac{1}{\beta_R}\int_{\Gamma_R}\phi^b_R\phi ds \right].
   \end{aligned}
 \end{equation}
Using the Poisson equation, the total energy can be rewritten as 
\begin{equation}\label{energy}
   E(\rho, \phi)=  \int_{\Omega}\bigg( \sum_{i=1}^s \rho_i\log \rho_i+\frac{1}{2}\epsilon(x)|\nabla \phi|^2 \bigg )d x-\int_{\Gamma_D}\epsilon(x) \phi_D^b \partial_n\phi ds +
       \frac{1}{2\beta_R} \int_{\Gamma_R} |\phi|^2 ds.
 \end{equation}
\begin{prop} Assume that $\phi^b$ does not depend on time, then the extended energy functional (\ref{energy}) 
satisfies 
\begin{equation}
\label{energydiss1}
\frac{d}{dt}E(\rho, \phi)(t) = -\int _{\Omega} \sum_{i=1}^s D_i(x)\rho_i|\nabla (\log \rho_i+z_i\phi) |^2 d x \leq 0, \quad   t>0,
\end{equation}
along the solution of (\ref{PNP1}).
\end{prop}
Recall that on $\Gamma_D$, the usual strategy for analysis is to  transform it to the case with zero boundary value for $\phi$. This way the modified energy would include an additional term called the external potential energy. 
For simplicity, we take $\phi_D^b=0$, so that we have the following result.

\begin{thm} \label{lb}(Lower bound and coercivity  of $E$ ) Let $\Omega$ be an open, bounded Lipschitz domain, and  $\phi^b$ be independent of time with $\phi_D^b=0$, $\beta_R>0$, and  $\epsilon(x) \geq a>0$. Then the energy of form 
\begin{equation}\label{energy+}
   E(\rho, \phi)=  \int_{\Omega}\bigg( \sum_{i=1}^s \rho_i\log \rho_i+\frac{1}{2}\epsilon(x)|\nabla \phi|^2 \bigg )d x + \frac{1}{2\beta_R} \int_{\Gamma_R} \phi^2 ds
 \end{equation}
 is bounded from below. Moreover, there exist constants $c_0, c_1>0$ such that  
\begin{equation}\label{coer}
E(\rho,\phi) \geq c_0\|\phi\|_{H^1}^2-c_1. 
\end{equation}
\end{thm}
\begin{proof} For $\rho_i\geq 0$, we have $\int_{\Omega}\sum_{i=1}^s\rho_i\log(\rho_i)\geq -s|\Omega|/e=: -c_1$. For the $\phi$-dependent part in $E$, we argue for all possible cases. For $\Gamma_D\not=\emptyset$ we have $\phi_D^b=0$; for purely Neumann's condition we have the additional condition $\int_{\Omega} \phi(x) dx  =0$, in either case we can apply the Poincaré inequality or the Poincaré--Wirtinger inequality to conclude 
$$
\|\phi\|^2_{L^2} \leq c^* \|\nabla \phi\|^2_{L^2}
$$
with constant $c^*$ depending on the geometry of $\Omega$, hence 
$$
E \geq -c_1 + \frac{a}{2}\|\nabla \phi\|^2_{L^2} \geq c_0\|\phi\|^2_{H^1}-c_1, \quad c_0=\frac{a}{4} \min\{1, \frac{1}{c^*}\}. 
$$
For the case $\partial \Omega=\Gamma_R\cup \Gamma_N$ with $\Gamma_R \not= \emptyset$, we have 
$$
E(\rho,\phi)\geq \frac{1}{2}\min\{a, \beta_R^{-1}\}\tilde E -c_1 
$$
with $\tilde{E}(\phi):= \int_\Omega |\nabla \phi|^2 dx +\int_{\partial\Omega} |\phi|^2 ds$.  We claim that 
$$
\tilde{E}(\phi)\geq c \|\phi\|_{H^1}^2 \quad \text{for some} \quad  c>0,
$$
which can be proved with a contradiction argument. Since otherwise we can assume  $\tilde{E}(\phi_n) <\frac{1}{n}\|\phi_n\|^2_{H^1}$.  Set $w_n=\phi_n/\|\phi_n\|_{H^1}$, then $w_n \in H^1(\Omega)$ with 
$$
\|w_n\|_{H^1}=1\; \text{and} \;  \|\nabla w_n\|^2_{L^2} < 1/n. 
$$
By the Rellich-Kondrachov theorem, we can extract a subsequence $\{w_{n_k}\}$ weakly converging to $w$ in $ H^1(\Omega)$ with $\nabla w_{n_k} \to 0$ weakly in $L^2(\Omega)$. This allows us to conclude  $w\in H^1$, and $\nabla w=0$.   From $\int_{\Gamma_R} |w_{n_k}|^2ds < 1/n_k$ and 
$$
\|w\|_{L^2(\Gamma_R)} \leq \| w_{n_k}\|_{L^2(\Gamma_R)} +\| w_{n_k}
- w\|_{L^2(\Gamma_R)} \leq 1/\sqrt{n_k} +C\|w-w_{n_k}\|_{H^1},
$$
we obtain $ w=0$ on $\Gamma_R$. Hence $w=0$ a.e., this is a contradiction. We complete this case by taking $c_0=  \frac{c}{2}\min\{a, \beta_R^{-1}\}$. 
\end{proof}

\subsection{Wasserstein distance and JKO scheme for multi-density}\label{PW-JKO} In order to derive a variational scheme for the PNP system with multi-density, we need to introduce a  Wasserstein-type distance. Motivated by the well-known characterization of the Wasserstein distance in a one-component fluid obtained by Benamou-Brenier \cite{BB00}, we consider to minimize a joint functional over the set 
\begin{equation}\label{Kset}
\begin{aligned}
  K: &  = \{\rho=(\rho_1, \cdots, \rho_s), u=(u_1, \cdots, u_s):\\
  & \quad \partial_t \rho_i +\nabla\cdot(\rho_i u_i)=0,  \quad (\rho_iu_i)\cdot \mathbf{n}  = 0 \quad \text{ on }   \partial \Omega \times [0, 1],\\ 
  & \quad \rho_i\in \mathcal{P}(\Omega), \quad \rho_i(x,0)=\rho_i^0(x), \quad \rho_i( x,1)=\rho_i^1(x) \}.  
\end{aligned}
\end{equation}
For the PNP system of two species $s=2$ with $D_i(x)=1$ and $\epsilon(x)=1$ considered in \cite{KMX17}, the distance inherited from the 2-Wasserstein distance is defined by 
$$
d^2(\rho^0, \rho^1)=\sum_{i=1}^2 W_2^2(\rho_i^0, \rho_i^1). 
$$
This is equivalent to the minimization of the joint functional:  
\begin{equation}\label{dd}
d^2(\rho^0, \rho^1): =\min_{(\rho, u)\in K} \sum_{i=1}^2 \int_0^1\int_{\Omega} |u_i|^2 \rho_idx dt.
\end{equation} 
Here $t$ is an artificial time and serves to characterize the optimal curve in the density space. 
 Following \cite{JKO98}, the authors in \cite{KMX17} constructed the following JKO scheme: Given a time step $\tau$, the scheme  defines a sequence $\rho^n$ as
 \begin{equation}\label{JKO}
 \rho^0=\rho^{in}, \quad \rho^{n+1}=\arg\min_{\rho\in [ \mathcal{P}(\Omega)]^2} \left\{ \frac{1}{2\tau}d^2( \rho^n, \rho)+\mathcal{E}(\rho) \right\}.
 \end{equation}
Here $\mathcal{E}$ is the total free energy, $d^2$ is the (squared) distance on the product space as defined in (\ref{dd}). 
One of the challenges in this program lies in handling the coupling terms, some intrinsic difficulties arise due to both the specific Poisson kernel and the system setting. Note that in \cite{KMX17} with $\epsilon(x)=1$,  the electrostatic potential $\phi$ in $E(\rho, \phi)$ is replaced by   
$$
\phi[\rho]= N*(f +\sum_{i=1}^2 z_i \rho_i), \quad x\in \Omega,
$$
so that $\mathcal{E}(\rho) =E(\rho, \phi[\rho])$. Here the kernel $N \sim C/|x|^{d-2}$ serves as a counterpart of the Green's function for the Newton potential in $\R^d$. Even with this treatment derivation of the corresponding Euler--Lagrange equations is quite delicate. We refer to \cite{KMX17} for further details. 

In order to extend the above JKO-type scheme to the present setting, we face two new difficulties: 
(i) $D_i(x)$ is no longer a constant, the kinetic energy corresponding to the squared distance cost needs to be modified; 
(ii) $\epsilon(x)$ is a general non-negative function, $\phi$ cannot be expressed explicitly in terms of $\rho$.  As for (i), we follow \cite{HLTW20} and consider a modified functional 
\begin{equation}\label{WD}
 d^2(\rho^0, \rho^1): =\min_{(\rho, u)\in K} \sum_{i=1}^s \int_0^1\int_{\Omega} D_i^{-1} |u_i|^2 \rho_idx dt.   
\end{equation}
As for (ii), the Poisson equation is treated as a constraint in the resulting minimization problem. For ease of presentation we define 
\begin{equation}\label{Aset}
\mathcal{A}:= \left\{(\rho, \phi): -\nabla \cdot(\epsilon(x) \nabla \phi)  = f(x)+\sum_{i=1}^{s}z_i\rho_i, \quad  \alpha \phi+\beta\frac{\partial \phi}{\partial \mathbf{n}}=\phi^b,  x\in \partial \Omega,  \rho\in [\mathcal{P}(\Omega)]^s \right\}. 
\end{equation}
For fixed $\rho^* \in [\mathcal{P}(\Omega)]^s$, and time step $\tau>0$ we set 
\begin{equation}\label{gt} 
   G_\tau(\rho, \phi) = 
\frac{1}{2\tau} d^2(\rho^*, \rho )+E(\rho, \phi), \quad (\rho, \phi)\in \mathcal{A}.
\end{equation}
In order to define a discrete sequence of approximate solutions using the minimizing scheme, we present a result on the existence of minimizers of $G_\tau$. 
To establish the uniqueness, we now prepare a technical lemma, with (iii) to be used later in the proof of Theorem \ref{thm3.1}.

\begin{lem}\label{lem2.1} Given $X^0, X^1$, let $X(\theta)=\theta X^0+(1-\theta)X^1$ for any $\theta\in (0, 1)$. \\
(i) If $X^0, X^1$ are vectors, then  
\begin{equation}\label{sq}
|X(\theta)|^2 -\theta |X^0|^2 -(1-\theta)|X^1|^2=-\theta(1-\theta)|X^1-X^0|^2.
\end{equation} 
(ii) If $X^0>0, X^1>0$ are scalars, then 
\begin{equation}\label{ent}
X(\theta) \log X(\theta) - \theta X^0 \log X^0 -(1-\theta)X^1 \log X^1
=-\theta(1-\theta)(X^1-X^0)^2 g(X^0, X^1; \theta), 
\end{equation} 
for some positive function $g$ depending on $X^0, X^1$ and $\theta$. \\
(iii) If $X^0>0, X^1>0$, $Y^0, Y^1$ are scalars, then 
\begin{equation}\label{Trans}
\frac{(Y(\theta))^2}{X(\theta)} - \theta \frac{(Y^0)^2}{X^0}  -(1-\theta)\frac{(Y^1)^2}{X^1} 
=-\theta(1-\theta)\frac{(X^1Y^0-X^0Y^1)^2}{X^0X^1X(\theta)}. 
\end{equation} 
\end{lem} 
\begin{proof}
We only prove  (ii); for (i) and (iii) can be verified by a direct calculation. Note that 
\begin{equation}\label{ID1}
 X(\theta) \log X(\theta)= \theta X^0 \log (X(\theta))+ (1-\theta)X^1\log (X(\theta)). 
\end{equation}
Taylor's  expansion of  $\log (X(\theta)$ at $X^0$ and $X^1$, respectively, gives 
\begin{equation*}\label{ID2}
 \log (\theta X^0+(1-\theta)X^1)=  \log ( X^0) +\frac{1}{X^0}(1-\theta)(X^0-X^1)-\frac{(1-\theta)^2(X^1-X^0)^2}{(\tilde{X}^0)^2}, 
\end{equation*}
where $\tilde{X}^0$ in between $X^0$ and $X(\theta)$, and 
\begin{equation*}\label{ID3}
 \log (\theta X^0+(1-\theta)X^1)=  \log ( X^1) +\frac{1}{X^1}+\theta(X^1-X^0)-\frac{\theta^2(X^1-X^0)^2}{(\tilde{X}^1)^2}, 
\end{equation*}
where $\tilde{X}^1$ in between $X^1$ and $X(\theta)$. Substituting these into the right hand side of (\ref{ID1}) leads to
\begin{align*}
 X(\theta) \log X(\theta)=& \theta X^0\log X^0-\theta (1-\theta)^2\frac{X^0}{(\tilde X^0)^2}(X^1-X^0)^2\\
 & + (1-\theta)X^1\log X^1 -  \theta^2 (1-\theta)\frac{X^1}{(\tilde X^1)^2}(X^1-X^0)^2,
\end{align*}
this completes the proof of (ii) by defining $g(X^0,X^1, \theta)=\frac{(1-\theta) X^0}{(\tilde{X}^0)^2}+\frac{\theta X^1}{(\tilde{X}^1)^2}>0$.
\end{proof}

\begin{thm} \label{thm2.3} (Existence of minimizers) Fix $\tau>0$, and $\rho^* \in [\mathcal{P}(\Omega)]^s$. Then the functional $G_\tau(\rho, \phi)$ admits a unique minimizer on $\mathcal{A}$. 
\end{thm} 
\begin{proof}  By Theorem \ref{lb}, $G_\tau$ is bounded from below on $\mathcal{A}$, hence there is a minimizing sequence $(\rho^k, \phi^k)$ and $\rho^k$ is tight and uniformly integrable. By the Dunford--Pettis Theorem
one may extract a subsequence such that
$\rho^k \to \rho$ in $L^1(\Omega)$, 
which together with $\rho^k \in [\mathcal{P}(\Omega)]^s$ ensure that $\rho \in [\mathcal{P}(\Omega)]^s$.
In addition,  $E(\rho, \cdot)$ is also coercive in $\phi$ because of (\ref{coer}), i.e.,    
$$
E(\rho, \phi)\geq c_0\|\phi\|^2_1-c_1.
$$ 
Hence one may extract a subsequence  such that $\phi^k \to \phi$ weakly in $H^1(\Omega)$. The weak $L^1$ lower semi-continuity (l.s.c.) of the squared Wasserstein distance can be easily adapted to the present case. 
The lower semicontinuity of $E$ with respect to weak convergence can be seen from the following inequality 
$$
E(\rho^k, \phi^k)\geq E(\rho, \phi)
+\int_{\Omega} \left[\sum_{i=1}^s {\rm ln} \rho_i(\rho_i^k-\rho_i)+\epsilon(x) \nabla \phi\cdot (\nabla \phi^k-\nabla \phi) \right]dx +
\frac{\alpha}{\beta}\int_{\partial \Omega} \phi(\phi^k-\phi)ds. 
$$
 Putting all these together  we claim that the limit is a minimizer. 

Finally, the uniqueness comes from the fact that the admissible set $\mathcal{A}$ is convex w.r.t. linear interpolation and that the total free energy is jointly strictly convex in $(\rho, \phi)$ on  $\mathcal{A}$.  More precisely, we argue as follows. Let $\theta \in (0, 1)$, then $\rho(\theta)=\theta \rho^0 +(1-\theta)\rho^1$ is a convex linear combination for $\rho^0$ and $\rho^1$. Let $\phi^0$ and $\phi^1$ be obtained from the Possion equation, corresponding to $\rho^0$ and $\rho^1$, respectively. Then 
$
\phi(\theta) =\theta \phi^0 +(1-\theta)\phi^1
$
must be the solution to the Poisson equation corresponding to $\rho(\theta)$. For the energy of form (\ref{energy+}), we evaluate $E(\rho(\theta), \phi(\theta))$ term by term to determine whether it is strictly convex. Using (\ref{ent}) for $\rho^l_i=X^l$ and (\ref{sq}) for
$X^l=\nabla \phi^l$ in $\Omega$, and 
(\ref{sq}) for $X^l=\phi^l$ on $\Gamma_R$, respectively, we obtain 
\begin{align*} 
E(\rho(\theta),\phi(\theta)) - \theta E(\rho^0, \phi^0) -(1-\theta)E(\rho^1, \phi^1) = -\theta(1-\theta)I
\end{align*} 
with 
$$
I=  \int_{\Omega}\bigg( \sum_{i=1}^s (\rho_i^1-\rho_i^0)^2g(\rho_i^0, \rho^1_i;\theta)  +\frac{1}{2}\epsilon(x)|\nabla (\phi^1-\phi^0)|^2 \bigg )d x + \frac{1}{2\beta_R} \int_{\Gamma_R} (\phi^1-\phi^0)^2 ds. 
$$
Convexity  of $E$ follows from $I\geq 0$. Actually this inequality is strict, unless $\rho^0 =\rho^1, \phi^0=\phi^1$, which can be derived from letting $I=0$. Hence $E(\rho, \phi)$ is strictly convex under two linear constraints. 
\end{proof} 

We are now ready to present a variational scheme formulation -- a JKO-type scheme for (\ref{PNP1F}): given time step $\tau>0$, recursively we define a sequence $\{\rho^n, \phi^n\}$ by 
\begin{equation}\label{JKO2}
\begin{aligned}
\rho^0= \rho^{\rm in}, \quad  (\rho^{n+1}, \phi^{n+1}) = &\arg \min_{(\rho, \phi)\in \mathcal{A}} \left\{ \frac{1}{2\tau} d^2(\rho^n, \rho )+E(\rho, \phi) \right\}.
\end{aligned}
\end{equation} 

\begin{thm}\label{thm2.4}(Solution properties of scheme (\ref{JKO2}) ) \\
(i) (Probability-preserving) If $\rho^n \in [\mathcal{P}(\Omega)]^s$, so is $\rho^{n+1}$; \\
(ii) (Unconditionally energy stability) the inequality 
$$
E(\rho^{n+1},\phi^{n+1}) +\frac{1}{2\tau}d^2(\rho^n, \rho^{n+1})\leq E(\rho^n, \phi^{n})
$$
holds for any $n\geq 0$. Furthermore, 
\begin{equation}\label{d2}
\sum_{n=0}^\infty d^2(\rho^n, \rho^{n+1})\leq 2\tau( E(\rho^0, \phi^{0}) - \inf_{(\rho, \phi)\in \mathcal{A}}E(\rho, \phi)).
\end{equation}
\end{thm}
\begin{proof} 
(i) The constraint $\mathcal{A}$ ensures that $\rho^n\in [\mathcal{P}(\Omega)]^s$ which is inherited from initial data; namely the method is both positivity and mass preserving. \\
(ii) 
From the definition of the minimizer, it follows
$$
E(\rho^{n+1},\phi^{n+1}) +\frac{1}{2\tau}d^2(\rho^n, \rho^{n+1})\leq E(\rho^n, \phi^{n}). 
$$
Here we used $d^2(\rho, \rho)=0$ for any $\rho \in  [\mathcal{P}(\Omega)]^s$. 
Finally, summation over $n$ yields (\ref{d2}). 
\end{proof} 

\subsection{Semi-discrete JKO scheme}
We proceed to obtain a computable formulation. Let $m_i=\rho_i u_i$,  the dynamic formulation of the distance $d^2(\cdot,\cdot)$ in (\ref{JKO2}) can be expressed as: given $\rho^n(x)$, we have 
\begin{equation}\label{JKO3}
\begin{aligned}
 (\rho^{n+1},\phi^{n+1})= &\arg \min_{(\rho, \phi) \in \mathcal{A},m} \left\{ \frac{1}{2\tau} \sum_{i=1}^s  \int_0^1\int_{\Omega}F(\rho_i, m_i)D^{-1}_i dx dt+E(\rho(\cdot, 1), \phi(\cdot, 1)) \right\},\\
 & \text{ s.t. } \quad   \partial_t \rho_i + \nabla\cdot (m_i) =0,\quad m_i\cdot \mathbf{n}  = 0, \quad x\in \partial \Omega, \; \rho(x,0)=\rho^n.
\end{aligned}
\end{equation} 
Here $t$ is an artificial time, and 
$$
F(\rho_i, m_i) 
=
\left\{
\begin{array}{ll}
\frac{ |m_i|^2}{\rho_i} & \text{if}\; \rho_i>0,\\
0 & \text{if}\; (\rho_i, m_i)=(0, 0),\\
+\infty & \text{otherwise}.
\end{array}
\right.
$$
The use of $m_i$ has enhanced the functional convexity in $m_i$ and made the transport constraint linear (see Breiner \cite{BB00}), yet causing difficulties for solutions near $\rho_i=0$. We shall prove for the fully discrete case positivity of $\rho_i^n$ is preserved for all $n$.  Another computational overhead with 
(\ref{JKO3}) is dealing with the artificial time $t\in [0,1]$ which is induced by the optimal transport flow.  
To overcome this issue, we follow \cite{LWC19} with a local approximation in the artificial time : approximate the derivative in $t$ in the constraint transport equation by a one step
difference and the integral in time in the objective function by a one term quadrature. We thus obtain the following scheme: 
\begin{equation}\label{JKO1}
\begin{aligned}
 (\rho^{n+1},\phi^{n+1})= &\arg\min_{(\rho,\phi)\in \mathcal{A},m} \left\{ \frac{1}{2\tau} \sum_{i=1}^s  \int_{\Omega} F(\rho_i, m_i)D^{-1}_i dx +E(\rho, \phi) \right\},\\
 & \text{ s.t. } \quad    \rho_i-\rho_i^n + \nabla\cdot (m_i) =0,\quad m_i\cdot \mathbf{n}  = 0, \quad x\in \partial \Omega. 
\end{aligned}
\end{equation} 
\begin{thm}\label{thm2.5} The positive minimizer of the variational
problem (\ref{JKO1}) is a first-order time consistent scheme for the PNP system.
\end{thm} 
\begin{proof}
Let (\ref{JKO1}) admit a minimizer with $\rho>0$. We can derive optimal conditions by the Lagrange multiplier method. Define the Lagrangian as 
\begin{align*}
 L(\rho,\phi,m, v, \xi) :=&\frac{1}{2\tau} \sum_{i=1}^s \int_{\Omega}  F(\rho_i, m_i)D^{-1}_i  dx  +E(\rho,\phi)+\int_{\partial \Omega} \xi (\alpha \phi +\beta \partial_n \phi -\phi^b)ds \\
 & +\sum_{i=1}^s\int_{\Omega}  v_i(\rho_i-\rho_i^n+\nabla \cdot m_i)dx+\int_{\Omega}v_{s+1}(f+\sum_{i=1}^sz_i\rho_i+\nabla \cdot (\epsilon(x)\nabla \phi)) dx.
\end{align*}
The optimality conditions for $x\in \Omega$ are 
\begin{align*}
    \frac{\delta L}{\delta \rho_i}=0 \quad & \text{ implies } & -\frac{1}{2\tau}\frac{||m_i||^2}{\rho_i^2}D^{-1}_i+\log (\rho_i)+1+\frac{1}{2}z_i\phi+v_i+z_iv_{s+1}=0,& \quad i=1,\cdots, s, \\
    \frac{\delta L}{\delta\phi}=0 \quad & \text{ implies } & \frac{1}{2}(f+\sum_{i=1}^sz_i\rho_i)+\nabla\cdot (\epsilon(x)\nabla v_{s+1})=0, \\
    \frac{\delta L}{\delta m_i}=0 \quad & \text{ implies } & \frac{1}{\tau}\frac{m_i}{\rho_i}D_i^{-1}-\nabla\cdot v_i=0,& \quad i=1,\cdots, s, \\
     \frac{\delta L}{\delta v_i}=0 \quad & \text{ implies } & \rho_i-\rho_i^n+\nabla\cdot m_i=0,& \quad i=1,\cdots, s, \\ 
    \frac{\delta L}{\delta v_{s+1}}=0 \quad & \text{ implies } & f+\sum_{i=1}^sz_i\rho_i+\nabla \cdot (\epsilon(x)\nabla \phi)=0.
\end{align*}
For $x\in \Omega$, we thus have 
$$
v_i=
\frac{1}{2\tau}\frac{||m_i||^2}{\rho_i^2}D^{-1}_i
-\log (\rho_i)-1-\frac{1}{2}z_i\phi-z_iv_{s+1}, \quad  m_i=\tau D_i\rho_i\nabla v_i
$$
and 
$$
\nabla\cdot (\epsilon(x)\nabla v_{s+1})=\frac{1}{2} \nabla\cdot (\epsilon(x)\nabla \phi). 
$$
On $\partial \Omega$, from integrating by parts in calculating $\delta L$ there remain the following boundary terms 
$$
 \int_{\partial \Omega}  \epsilon(x) \delta (\partial_n\phi) v_{s+1} ds - \int_{\partial \Omega}  \epsilon(x) \delta \phi \partial_n v_{s+1}ds+
  \int_{\partial \Omega} v_i  \delta m_i\cdot \mathbf{n}ds,  
$$
where the last  term vanishes due to the constraint 
$
m_i\cdot \mathbf{n}=0.
$
In addition, we need also consider terms arising from  
$$
\delta B+\delta \int_{\partial \Omega} \xi (\alpha \phi +\beta \partial_n \phi -\phi^b)ds.
$$
Upon careful regrouping, we have two cases to distinguish:  \\
(i) for $\beta\not= 0$, the correction term $B$ in the energy (\ref{Intro-Energy}) is given by
$$
B=\frac{1}{2\beta} \int_{\partial \Omega} \phi^b \phi ds.
$$
We obtain 
$$
\epsilon(x) v_{s+1} +\beta \epsilon(x) \xi =0, \quad 
-\epsilon(x)\partial_n v_{s+1} +\alpha \xi +\frac{1}{2\beta} \phi^b=0, \quad \text{on}\; \partial \Omega; 
$$
(ii) For $\beta=0$, \\
The correction term $B$ in the energy (\ref{Intro-Energy}) is given by
$$
B=-\frac{1}{2\alpha} \int_{\partial \Omega} \epsilon(x)\phi^b \partial_n\phi ds.
$$
from which we have 
$$
\epsilon(x)v_{s+1} -\frac{1}{2\alpha} \epsilon(x) \phi^b +\beta \xi=0, \quad 
-\epsilon(x)\partial_n v_{s+1}  +\alpha \xi=0, \quad \text{on}\; \partial \Omega.
$$
These ensure that we always have 
$$
\alpha v_{s+1}+\beta \partial_n v_{s+1}=\frac{1}{2}\phi^b \quad \text{on}\; \partial \Omega.
$$
Take $\psi=\frac{1}{2}\phi-v_{s+1}$ we have 
$$
\nabla \cdot (\epsilon(x)\nabla \psi)=0, x\in \Omega; \alpha \psi +\beta \partial_n \psi=0 \; \text{on}\; \partial \Omega.
$$
By the uniqueness of the Poisson problem we conclude $\psi\equiv 0$ or 
$\psi={\rm cost}$ if $\alpha=0$,  i.e.,   
$$
v_{s+1}\equiv \frac{1}{2}\phi+cost.  
$$
Combing the above we have the following update  
$$
\rho_i=\rho^n_i+\tau \nabla \cdot \left(D_i\rho_i \nabla (\log(\rho_i)+z_i\phi)\right)+O(\tau^2).
$$
This says scheme (\ref{JKO1}) is a first order time discretization of the PNP system (\ref{PNP1F}).
\end{proof}
\begin{rem} A natural question arises: is the discrete transport still preserves positivity of $\rho_i$.  We shall address this issue for the fully discrete scheme, for which positivity propagation is rigorously established in Theorem \ref{thm3.2}.  
\end{rem}

\section{Numerical method}
In this section, we detail the spatial discretization. The underlying principle for spatial discretization is to preserve the structure of Wasserstein metric tensor in the discrete sense. 


\subsection{Spatial discretization}
We only consider the discretization in one dimensional setting. Let $\Omega=[a, \ b]$ be the computational domain partitioned into $N$ cells $I_j=[x_{j-\frac{1}{2}}, \ \ x_{j+\frac{1}{2}}]$,  with mesh size $h=(b-a)/N$ and cell center at $x_j=x_{j-\frac{1}{2}}+\frac{1}{2}h$, $j\in \{1,2,\cdots, N\}.$  Let numerical solution be 
$\{\phi_{j}\}_{j=1}^{N}$, $\{\rho_{ij}\}_{j=1}^N$, and $ \{m_{i, j+1/2}\}_{j=1}^{N-1}$ on two grids $x_j$ and $x_{j+1/2}$, respectively.  We define the difference operator by 
$$
(D_h v)_{j+1/2}:=\frac{v_{j+1}-v_j}{h}, \quad (d_h v)_j =\frac{v_{j+1/2}-v_{j-1/2}}{h}
$$
and average operator by 
$$
\hat v_j=\frac{v_{j+1/2}+v_{j-1/2}}{2}. 
$$
We also use $\epsilon_{j+1/2}=\epsilon(x_{j+1/2})$, $f_j=f(x_j)$, and $D_{ij}=D_i(x_j)$.

The transport constraint is discretized with central difference
in space as follows: 
\begin{equation}\label{dt}
\rho_{ij}-\rho^n_{ij}+d_h(m_i)_j=0, 
\end{equation} 
and the zero boundary conditions  $m_{i,1/2}=m_{i,N+1/2}=0$ are applied.  

For the Possion equation, we consider the Robin boundary condition at both ends, other types of boundary conditions can be handled in same fashion.  We introduce two ghost values $\phi_0$ and $\phi_{N+1}$ for conveniently approximating the boundary condition (\ref{BPO}) with center differences: 
\begin{equation}\label{GP}
\frac{\phi_0+\phi_1}{2}-\beta_a \epsilon(a) \frac{\phi_1-\phi_0}{h} = \phi^b(a), \quad  \frac{\phi_{N+1}+\phi_N}{2}+\beta_b \epsilon(b) \frac{\phi_{N+1}-\phi_N}{h}=\phi^b(b).
\end{equation}
This together with the center difference approximation of the Poisson equation  gives a coupled linear system:  
\begin{equation}\label{dp}
\begin{aligned}
& (h+2\beta_a \epsilon(a)) \phi_0 + (h-2\beta_a \epsilon(a)) \phi_1-2h \phi^b(a)=0, \\  
& -d_h(\epsilon D_h \phi)_j -f_j - \sum_{i=1}^s z_i\rho_{ij}=0,  \quad j=1,\cdots, N,\\
&  (h-2\beta_b\epsilon(b))\phi_N + (h+2\beta_b\epsilon(b)) \phi_{N+1}-2h \phi^b(b)=0.  
\end{aligned}
\end{equation}
We denote such linear constraint by  $L_h(\phi, \rho)=0$. 
The objective function then writes as 
\begin{equation}\label{DOB+}
\begin{aligned}
 \mathcal{F}_h(\rho, m, \phi) 
 = & \frac{h}{2\tau} \sum_{j=1}^N \sum_{i=1}^{s} \frac{   \hat {m}^2_{i,j}      }{\rho_{i,j}} D^{-1}_{i,j } +h\sum_{j=1}^{N} \left( \sum_{i=1}^s \rho_{i,j}\log \rho_{i,j} +\frac{\epsilon_j}{8h^2} (\phi_{j+1}-\phi_{j-1})^2\right)\\
& +\frac{1}{8\beta_a}(\phi_0+\phi_1)^2+\frac{1}{8\beta_b}(\phi_N+\phi_{N+1})^2,
     \end{aligned}
\end{equation}
which is a second order spatial approximation of the objective functional in (\ref{JKO1}). 

To formulate an admissible set for the discrete minimization problem, let the discrete probability distribution set be: for $\delta>0$
$$
P_{h, \delta}=\left\{ (\rho_1, \cdots, \rho_N): \quad  \rho_j \geq \delta, \quad  h\sum_{j=1}^N \rho_j=1 \right\}.   $$
Then the constraint set for $(\rho, \phi)$ becomes 
$$
\mathcal{A}_{h, \delta}=\{(\rho, \phi): \quad \rho \in [P_{h, \delta}]^s, \quad L_h(\phi, \rho)=0\}. 
$$
Thus the admissible set for all $(\rho, m, \phi)$ collectively can be written as 
$$
V_{h, \delta}^n=\{ (\rho, m, \phi): \quad 
\rho_{ij}-\rho_{ij}^n +d_h (m_i)_j=0, \quad (\rho, \phi)\in \mathcal{A}_{h, \delta}\}
$$
with $m_{i,1/2}=m_{i,N+1/2}=0$. Thus we have 
$$
V_{h, \delta}^n \subset  \mathbb{R}^{s(2N-1)+N+2}.
$$
The one time update with the fully discrete scheme is to find 
\begin{equation}\label{fKO}
\begin{aligned}
\rho^{n+1}= \arg &\min _{u \in V_{h, \delta}^n} \bigg\{  \mathcal{F}_h(u)  \bigg\}, \quad u:=(\rho, m, \phi).
\end{aligned}
\end{equation}
\begin{thm} \label{thm3.1}(Unique solvability) 
Fix $\tau>0, h>0$ and $\{\rho_{i}^n \in P_{h,\delta}\}_{i=1}^s$ for some $\delta >0$. Then the function $\mathcal{F}_{h}(\rho,m,\phi)$ admits a unique minimizer in $V_{h,\delta}^n \subset  \mathbb{R}^{s(2N-1)+N+2}$.
\end{thm}
\begin{proof} The proof proceeds in two steps: \\
{\sl Step 1 (Admissible set is non-empty and convex)}  The conservative form of the transport constraint ensures that we always have 
$$
h\sum_{j=1}^N \rho_{ij}=1 \quad i\in [s]. 
$$
For fixed $\delta >0$, take $\rho_{ij} \geq  \delta$, we can uniquely determine $m$ by  
\begin{equation}\label{DT}
 m_{i, j+1/2}=\frac{1}{h}\sum_{l=1}^j(\rho_{il}-\rho_{il}^n),   
\end{equation}
for $j=1, \cdots, N-1$. From the linear system $L(\phi\;, \rho)=0$ we obtain a unique $\phi=(\phi_0, \cdots, \phi_{N+1})$ in terms of $f_j$ and $\rho_{ij} \geq \delta$, since its coefficient matrix is tridiagonal, and diagonally dominated.  Hence the admissible set $V_{h, \delta}^n$ is non-empty. The fact that both the transport constraint and $L(\phi, \rho)=0$ are linear implies that the set $V_{h, \delta}^n$ is convex in $\mathbb{R}^{s(2N-1)+N+2}$.  \\ 
{\sl Step 2 (Objective function is strictly convex under constraints)}\\
With $u=(\rho, m, \phi)$, for any $u^0, u^1 \in V_{h,\delta}$  and $\theta \in (0, 1)$,  $u(\theta)=\theta u^0 +(1-\theta)u^1$ is a convex linear combination of $u^0$ and $u^1$. In addition, as argued in the proof of Theorem \ref{thm2.3},  
we have $u(\theta)\in V_{h,\delta}$. We now show the convexity of $\mathcal{F}_{h}(u)$ by directly calculating 
$$
\mathcal{F}_h(u(\theta)) - \theta \mathcal{F}_h(u^0))- (1-\theta)\mathcal{F}_h(u^1) =-\theta(1-\theta)(I_1 +I_2+I_3), 
$$
where applying Lemma \ref{lem2.1} to each term $I_i$, we have 
\begin{align*}
I_1=&\frac{h}{2\tau} \sum_{i=1}^s\sum_{j=1}^N\frac{(\rho^1_{i,j}\hat m^0_{i,j}-\rho^0_{i,j}\hat m^1_{i,j})^2}{\rho^0_{i,j}\rho^1_{i,j}\rho(\theta)_{i,j}}\geq 0,\quad \text{by}\;  (iii) \; \text{of Lemma} \; \ref{lem2.1}\\
I_2=& h\sum_{i=1}^s\sum_{j=1}^N g_{i,j}(\rho^0_{i,j},\rho^1_{i,j}, \theta)(\rho^0_{i,j}-\rho^1_{i,j})^2\geq 0, \;\; \text{by}\;  (ii) \; \text{of Lemma} \; \ref{lem2.1}\\
I_3=& \frac{1}{8h}\sum_{j=1}^N\epsilon_j[(\phi^0_{j+1}-\phi^0_{j-1})-(\phi^1_{j+1}-\phi^1_{j-1})]^2 \quad \text{by}\;  (i) \; \text{of Lemma} \; \ref{lem2.1}\\
&+ \frac{1}{8\beta_a} [(\phi_0^0+\phi_1^0)-(\phi_0^1+\phi_1^1)]^2
    + \frac{1}{8\beta_b} [(\phi^0_N+\phi^0_{N+1})-(\phi^1_{N}+\phi^1_{N+1})]^2\geq 0.
\end{align*}
Convexity of $\mathcal{F}_h$ follows from $I_1+I_2+I_3\geq 0$. To establish strictly convexity we only need to show $I_1+I_2+I_3= 0$  must lead to $u^0=u^1$.  We argue as follows. 

Clearly the equality holds only when $I_1=I_2=I_3=0$.  From $I_2=0$ it follows $\rho^0=\rho^1$. This when combined with $I_1=0$ implies $\hat m^0_{i,j}=\hat m^1_{i,j}$, which  together with  $m^l_{i,1/2}=m^l_{i,N+1/2}=0$ yields  $m^0=m^1$. Finally we show $\phi^0=\phi^1$ must also hold. Set $\xi_j=\phi_j^0-\phi_j^1$ for $j=0,\cdots, N+1,$ then $I_3=0$ corresponds to the system of linear equations $\xi_0+\xi_1=0,$ $\xi_N+\xi_{N+1}=0$ and 
$\xi_{j+1}-\xi_{j-1}=0$, for $j=1, \cdots, N$. This obviously admits non-zero solutions. 
From the constraint for $\phi$ near the boundary we have 
$$
\phi_0^0+\phi_1^0=2\phi^b(a)+\frac{\beta_a}{h}\epsilon(a) (\phi^0_1-\phi^0_0), \quad \phi_0^1+\phi_1^1=2\phi^b(a)+\frac{\beta_a}{h}\epsilon(a) (\phi^1_1-\phi^1_0),
$$
this implies $\xi_0+\xi_1=\frac{\beta_a}{h}\epsilon(a)(\xi_1-\xi_0)$. Using also $\xi_0+\xi_1=0$, we can conclude $\xi_0=\cdots =\xi_{N+1}=0$, therefore $\phi^0=\phi^1.$
Hence $\mathcal{F}_h(u)$ is strictly convex on $V_{h,\delta}$.
\end{proof} 

The last issue is to find a threshold for $\delta$ so to ensure that solution positivity for the PNP system is propagated at all time steps.
 \begin{thm}\label{thm3.2}(Positivity propagation) 
 There exists $\delta_0>0$ such that the minimizer does not touch the boundary of
 $V_{h, \delta}^n$ for all $0<\delta \leq \delta_0$.  This implies that $\rho^n>0$ for all $n>0$ as long as 
 $\rho^0>0$. 
 \end{thm} 
 \begin{proof} We use a contradiction
argument: suppose there exists a minimizer 
$u^*$ to the optimization problem (\ref{fKO}) touching the boundary of 
$V_{h,\delta}^n$ at some grid points $j_1 < \cdots < j_k$ with $1 \leq  k \leq N-1$ for $\rho_i$, that is
$$
\rho_{i, j_1}^*=\cdots =\rho_{i, j_k}^*=\delta.
$$
From $h \sum_{j=1}^N \rho_{i, j}^*=1$, we see that $\delta <\frac{1}{b-a}$. Since $\mathcal{F}_h$ is convex and differentiable, we only need to find $u \in \mathcal{A}_{h,\delta} $ such that 
\begin{equation}\label{nn}
\nabla \mathcal{F}_h(u^*)\cdot (u-u^*)<0. 
\end{equation}
Note that both $m$ and $\phi$ can be uniquely determined by $\rho$ from the constraints, it suffices to first choose $\rho$ and then express all components of $u$ in terms of $\rho$. Let $\rho_{i, j_{k+1}}^*$ be the maximum component in  vector $\rho^*_i$, using $h \sum_{j=1}^N \rho_{i, j}^*=1$ we thus have 
\begin{equation}\label{rs}
\frac{1}{b-a} < \rho^*_{i, j_{k+1}} <\frac{1}{h}=\frac{N}{b-a}. 
\end{equation}
Without loss of generality, we assume $j_{k+1}>j_k$, and 
\begin{equation}\label{rs+}
\rho_{i, j}^* \geq \delta +r_p(h), \quad j_p<j<j_{p+1}, \quad p=1, \cdots, k, 
\end{equation}
where $r_p(0)=0$ and $r_p(h)>0$ for $h>0$ small. This can be justified by approximation for sufficiently small $h$.  Fix $h>0$, we take for $0<\gamma< \frac{1}{k}(\frac{1}{b-a}-\delta)$,
$$
\rho_{l,j} =\left\{ \begin{array}{ll}
    \delta +\gamma,  & l=i, \; j=j_1, \cdots, j_k, \\
    \rho^*_{i, j_{k+1}}- \gamma k,  & l=i, j=j_{k+1},\\
    \rho^*_{l,j}, & \text{else}.
\end{array}
\right.
$$
Hence $\tilde u=u-u^*$ can be determined by 
$$
\tilde \rho_{l,j}= \rho_{l, j}-\rho^*_{lj} =\left\{ \begin{array}{ll}
    \gamma, & l=i, \; j=j_1, \cdots, j_k, \\
    -\gamma k,  & l=i, j=j_{k+1},\\
    0, & \text{else}.
\end{array}
\right.
$$
Using $\tilde m=m-m^*$ and formula (\ref{DT}) for both $m$ and $m^*$,   
we have  
\begin{equation}\label{DT+}
\tilde m_{l,j+1/2} =\frac{1}{h}
\sum_{p=1}^j\tilde \rho_{l,p} = \left\{ \begin{array}{ll}
    \frac{1}{h}b_j \gamma, & l=i, \; j_1\leq j\leq  j_{k+1}-1,\\
    0, & \text{else},
\end{array}
\right.
\end{equation}
for $0\leq b_j \leq k$.  Hence 
$$
0\leq \hat{\tilde m}_{i, j} \leq \frac{k\gamma}{h}, \quad j_1\leq j\leq  j_{k+1}.
$$
For $\tilde \phi=\phi-\phi^*$, using (\ref{dp}) for both $\phi$ and $\phi^*$, we obtain $A \tilde \phi=[0, z_ih^2\tilde \rho_i, 0]^\top$, where the coefficient matrix $A$ is non-singular, more precisely, $\tilde \phi$ solves  
\begin{align*}
& (h+2\beta_a \epsilon(a)) \tilde \phi_0 + (h-2\beta_a \epsilon(a)) \tilde \phi_1 =0, \\  
& -\epsilon_{j-1/2} \tilde \phi_{j-1} + 2 \hat \epsilon_j \tilde \phi_j -  \epsilon_{j+1/2} \tilde \phi_{j+1} = h^2 z_i  \tilde \rho_{i, j}  \quad j=1,\cdots, N,\\
&  (h-2\beta_b\epsilon(b))\tilde \phi_N + (h+2\beta_b\epsilon(b)) \tilde \phi_{N+1}=0.  
\end{align*} 
The solution of this linear system can be expressed as  
$$
\tilde \phi_l=\gamma  h^2(c_l- k d_l)z_i, \quad l=0, 1, \cdots, N+1
$$
for some $c_l, d_l$ depending on the coefficients in the above system. The above preparation yields
\begin{align*}
& \nabla \mathcal{F}_h(u^*)\cdot (u-u^*)  = \nabla \mathcal{F}_h(u^*)\cdot \tilde u \\
\qquad & = \sum_{l=1}^s\sum_{j=1}^N \partial_{\rho_{l,j}}\mathcal{F}_h(u^*) \tilde \rho_{l, j} + \sum_{l=1}^s\sum_{j=1}^{N-1} \partial_{m_{l, j+1/2}} \mathcal{F}_h(u^*)\tilde m_{l, j+1/2} + \sum_{j=0}^{N+1} \partial_{\phi_{j}}\mathcal{F}_h(u^*)\tilde \phi_{j}\\
& =\gamma \left[
\sum_{p=1}^k \partial_{\rho_{i,j_p}}\mathcal{F}_h(u^*) 
-k \partial_{\rho_{i, j_{k+1}}} \mathcal{F}_h(u^*) \right]+\sum_{j=j_1}^{j_{k+1}-1} 
\partial_{m_{i, j+1/2}} \mathcal{F}_h(u^*)\tilde m_{i, j+1/2}
+\sum_{j=0}^{N+1} \partial_{\phi_{j}}\mathcal{F}_h(u^*)\tilde \phi_{j}\\
&=: I_1+I_2+I_3.
\end{align*}
In order to estimate $I_1, I_2, I_3$ we also need to bound $u^*$ in terms of $\rho^*$.
From (\ref{DT}) and (\ref{rs}) we have 
$$
|m^*_{i, j+1/2}|\leq \frac{1}{h}\sum_{l=1}^j \rho^*_{il}\leq \frac{Nj}{(b-a)h} \Rightarrow |\hat m^*_{i, j}|\leq \frac{N^2}{(b-a)h}=Nh^{-2}.
$$
For $\phi^*$ satisfying $A\phi^*=h[2\phi^b(a), h(f+\sum_{i=1}^s z_i \rho^*_{i}), 2\phi^b(b)]^\top$, we have 
$$
|\phi^*|\leq h|A^{-1}|[2\phi^b(a), h(f+\sum_{i=1}^s z_i \rho^*_{i}), 2\phi^b(b)]^\top|=:C^*_\phi.
$$
We proceed as follows: from the definition of the objective function 
\begin{align*}
 \mathcal{F}_h(u) 
 = & \frac{h}{2\tau} \sum_{j=1}^N \sum_{i=1}^{s} \frac{   \hat {m}^2_{i,j}      }{\rho_{i,j}} D^{-1}_{i,j } +h\sum_{j=1}^{N} \left( \sum_{i=1}^s \rho_{i,j}\log \rho_{i,j} +\frac{\epsilon_j}{8h^2} (\phi_{j+1}-\phi_{j-1})^2\right)\\
& +\frac{1}{8\beta_a}(\phi_0+\phi_1)^2+\frac{1}{8\beta_b}(\phi_N+\phi_{N+1})^2,
\end{align*}
given in (\ref{DOB+}) we have 
\begin{align*}
I_1& = \gamma 
\sum_{p=1}^k \left[ -\frac{h}{2\tau} \cdot \frac{(\hat m^*_{i, j_p})^2}{(\rho_{i,j_p}^*)^2}D_{i,j_p}^{-1}+h(1+\log \rho_{i,j_p}^*)\right] 
-\gamma k \left[ -\frac{h}{2\tau} \cdot \frac{(\hat m^*_{i, j_{k+1}})^2}{(\rho_{i,j_{k+1}}^*)^2}D_{i,j_{k+1}}^{-1}+h(1+\log \rho_{i,j_{k+1}}^*)\right]\\
&=-\frac{\gamma h}{2\tau \delta^2}\sum_{p=1}^k (\hat m^*_{i, j_p})^2 D_{i,j_p}^{-1}
+\gamma h k {\rm log} \delta 
+ \frac{\gamma k h}{2\tau } \frac{(\hat m^*_{i, j_{k+1}})^2}{(\rho^*_{i, j_{k+1}})^2} D_{i,j_{k+1}}^{-1}
-\gamma h k {\rm log} \rho^*_{i, j_{k+1}} \\
& \leq -\frac{\gamma h}{2\tau \delta^2}\sum_{p=1}^k (\hat m^*_{i, j_p})^2 D_{i,j_p}^{-1} + 
\gamma h k {\rm log} \delta +\frac{\gamma kN^4}{2\tau h} D_{i,j_{k+1}}^{-1} +\gamma hk {\rm log} (b-a)\\
&=  -\frac{\gamma h}{2\tau \delta^2}\sum_{p=1}^k (\hat m^*_{i, j_p})^2 D_{i,j_p}^{-1}
 +\frac{\gamma kN^4}{2\tau h} D_{i,j_{k+1}}^{-1} +\gamma hk {\rm log} \delta (b-a).
\end{align*}
Next, we estimate $I_2$:
\begin{align*}
I_2 &=\frac{h}{2\tau}\sum_{j=j_1}^{j_{k+1}-1} \left(  \frac{\hat{m}^*_{i,j}}{\rho^*_{i,j}} D_{i,j}^{-1}+\frac{\hat{m}^*_{i,j+1}}{\rho^*_{i,j+1}} D_{i,j+1}^{-1} \right)\tilde m_{i,j+1/2}\\
& =\frac{h}{\tau} 
\sum_{j=j_1}^{j_{k+1}-1}   \frac{\hat{m}^*_{i,j}}{\rho^*_{i,j}} D_{i,j}^{-1} 
\hat{\tilde m}_{i,j}
+ \frac{h}{2\tau} 
\left(-\frac{\hat{m}^*_{i,j_1}}{\rho^*_{i,j_1}} D_{i,j_1}^{-1}
\tilde m_{i,j_1-1/2} +\frac{\hat{m}^*_{i,j_{k+1}}}{\rho^*_{i,j_{k+1}}} D_{i,j_{k+1}}^{-1}
\tilde m_{i,j_{k+1}-1/2} 
\right)\\
& = \frac{h}{\tau} 
 \sum_{p=1}^k \frac{\hat{m}^*_{i,j_p}}{\delta} D_{i,j_p}^{-1} 
\hat{\tilde m}_{i,j_p} 
 +\frac{h}{\tau} 
\sum_{p=1}^k\sum_{j=j_p+1}^{j_{p+1}-1}  \frac{\hat{m}^*_{i,j}}{\rho^*_{i, j}} D_{i,j}^{-1} 
\hat{\tilde m}_{i,j} +\frac{h}{2\tau}\cdot \frac{ \hat{m}^*_{i,j_{k+1}}}{ \rho^*_{i,j_{k+1}}} D_{i,j_{k+1}}^{-1} 
\tilde m_{i,j_{k+1}-1/2}  \\ 
& \leq \frac{h}{\tau}\cdot \frac{k\gamma}{h}  \left(
\sum_{p=1}^k\frac{|\hat{m}^*_{i,j_p}|}{\delta}D_{i,j_{p}}^{-1}  +
 \sum_{p=1}^k\sum_{j=j_p+1}^{j_{p+1}-1} D_{i,j}^{-1}\frac{Nh^{-2}}{\delta +r_p (h)}+ \frac{N^2}{h} D_{i,j_{k+1}}^{-1}  \right)  \\
& \leq \frac{k\gamma}{\tau}\left(  \frac{ \eta}{2 \delta^2}\sum_{p=1}^k (\hat{m}^*_{i,j_p})^2 D_{i,j_{p}}^{-1}  +  \frac{1}{2 \eta}\sum_{p=1}^k D_{i,j_{p}}^{-1}
+\frac{N }{ h^2}  \sum_{p=1}^k\sum_{j=j_p+1}^{j_{p+1}-1} \frac{D_{i,j}^{-1}}{r_p(h)} + \frac{N^2}{h} D_{i,j_{k+1}}^{-1} \right), \quad \forall \eta>0. 
\end{align*}
 Take $\eta$ so that $k\eta =h$, we have 
$$
I_2 \leq \frac{\gamma h}{2\tau \delta^2}\sum_{p=1}^k (\hat m^*_{i, j_p})^2D_{i,j_{p}}^{-1} +C_1, \quad C_1:=\frac{k^2\gamma }{2 \tau h}\sum_{p=1}^k D_{i,j_{p}}^{-1}
+\frac{N k\gamma  }{ \tau h^2}  \sum_{p=1}^k\sum_{j=j_p+1}^{j_{p+1}-1} \frac{D_{i,j}^{-1}}{r_p(h)} +  \frac{N^2k \gamma }{\tau h} D_{i,j_{k+1}}^{-1} .
$$
Note that 
$$
\partial_{\phi_j} \mathcal{F}_h(u^*)=\frac{1}{4h} \left[\epsilon_{j-1}(\phi^*_j-
\phi^*_{j-2})+\epsilon_{j+1}(\phi^*_j-
\phi^*_{j+2}) \right] \quad j=2, \cdots, N-1,
$$
this together with derivatives involving boundary terms allows us to estimate $I_3$:  
\begin{align*}
I_3 & \leq |\partial_{\phi} \mathcal{F}_h(u^*)|\cdot |\tilde \phi|\\
& \leq \frac{1}{h} (|\epsilon|+\beta_a^{-1}+\beta_b^{-1})C_\phi^* \cdot \gamma |z|(|c|+k|d|)h^2 =C_0C^*_\phi h.
\end{align*}
For $\delta < \frac{1}{2(b-a)}$, we can take $\gamma k=\frac{1}{2(b-a)}$ such that $k \gamma < 1/(b-a)-\delta$ still holds. Hence   
\begin{align*} 
I_1+I_2 +I_3& \leq  \gamma hk {\rm log} \delta (b-a) +\frac{\gamma kN^4}{2\tau h} D_{i,j_{k+1}}^{-1}+C_1  +C_0C^*_\phi h \\
& =\frac{1}{2N} {\rm log} \delta (b-a) +\frac{N^3}{4\tau h^2}D_{i,j_{k+1}}^{-1} +C_1+C_0C^*_\phi h <0 
\end{align*} 
provided $\delta < \delta_0$ with  
$$
\delta_0: = \frac{1}{b-a} \min\left\{\exp \left(-\frac{N^4}{2\tau h^2}D_{i,j_{k+1}}^{-1} -2NC_1-2C_0C_\phi^*Nh \right), \frac{1}{2}\right\}.  
$$
This gives (\ref{nn}) as we intended to show. Such contradiction allows us to conclude that a minimizer at nth step can only occur in the interior of
$
V_{h, \delta_0}^n
$
for some $\delta_0>0$.  In order to show such solution positivity can propagate, we start from $\rho^0>0$. Based on the above conclusion we recursively have  
$$
\rho^{n+1} \in V_{h, \delta}^n \subset V_{h, \delta_0}^n.
$$
This completes the proof. 
\end{proof}

\section{Optimization algorithms} In this section, we discuss numerical techniques for solving the constrained optimization problem  (\ref{fKO}).  
Let $u = (\rho, m, \phi)$, (\ref{fKO}) can be written as
\begin{equation}\label{finalp}
 \min_u \mathcal{F}_h(u), s.t. \quad Au = b, \quad S u \geq \delta,   
\end{equation}
Where $\mathcal{F}_h(u)$ is defined in (\ref{DOB+}), $Au=b$ is the linear system corresponding to the constraints (\ref{dt}) and (\ref{dp}), and $S$ is the selection matrix that only selects $\rho$ component in $u$. 

A simple method to solve (\ref{finalp}) is to do the following update:
$$
\tilde u^{n+1}=u^n-\eta G \nabla_u  \mathcal{F}_h(u^n), 
$$
with the matrix 
$$
G=I-A^\top (AA^\top)^{-1}A, 
$$
so that $ A \tilde u^{n+1}=b$ if $Au^n=b$. 
 One then applies the projection 
$$
u^{n+1}=\Pi(\tilde u^{n+1}) 
$$
so that $\rho_{ij}^{n+1}\geq \delta$. \\
\begin{algorithm}[H]
\SetAlgoLined
\KwIn{$A, b$, $u^n$, $Iter_{max}$,  and $\epsilon$.}
\KwOut{$u^{n+1}$ }
 initialization\;
 $G=I-A^T(AA^T)^{-1}A$, \quad $u^{(0)}=u^n$. 
 
 \For{$k=1:Iter_{max}$}{
 \begin{itemize}
  \item Compute the update direction by
  $$
 \Delta u=-G\nabla_u\mathcal{F}_h(u^{(k-1)})
  $$
  \item Use backtracking to determine step size $\eta$\;
  \item Update to get 
$$
\tilde{u}=u^{(k-1)}+\eta \Delta u
$$
\item Projection $u^{(k)}=\Pi (\tilde u)$\;
  \end{itemize}
  \If{$||Au^{(k)}-b||+||\eta \Delta u||\leq \epsilon$}{
   Stop the iteration\;
 }}
 $u^{n+1}=u^{(k)}$. 
 \caption{ PG Algorithm}
\end{algorithm}

The positivity propagation property stated in Theorem \ref{thm3.2} ensures that $Su\geq \delta $ will be fulfilled by the scheme as long as $\rho^0 \geq \delta$ for $\delta$ suitably small.  
Hence in our numerical tests the projection step is not enforced, where we select 
$$
\delta=\max\{ \min\{h^2, \tau\}, \min \{\rho^{in}_{i}(x_j) \} \}>0.
$$

In summary, the numerical solutions $\rho^{n}_{i,j}$ and $\phi^{n}_{j}$ are updated with the algorithm: 

\begin{algorithm}[H]
\SetAlgoLined
\KwIn{$\rho^{in}_i(x)$, final time $T$, and discretization parameters $h$, $\tau$, $\delta>0$} 
\KwOut{$\rho^{n}_{i,j}$, $\phi^n_{j}$ for $n=1,\cdots, T/\tau$. }
 initialization:  $u^0=(\rho^0, m^0, \phi^0)$ with \\
 $\rho^0_{ij}=\max\{\rho^{in}_i(x_j), \delta\}$. \\ 
 $m^0_{i,j}=0$, and $\phi^0_{j}$ is obtained by solving (\ref{GP}) with $\rho^0_{i,j}$. \\
 \For{$n=1:T/\tau$}{
 $ \rho^{n+1}= \arg \min _{u \in V_{h, \delta}^n} \bigg\{  \mathcal{F}_h(u)  \bigg\}$ with Algorithm 1.
}
 \caption{ Algorithm for the fully discrete scheme}
\end{algorithm}
\begin{rem}
One may also apply the Primal-Dual Interior-Point algorithm (PDIP) \cite[Chapter 19]{NW06}) to solve the minimization problem in Algorithm 2, which helps to enforce the positive lower bound for densities, but it is much more expensive than the PG method, see a comparison in Table 3.  
\end{rem}


\section{Numerical tests}
In this section, we present  a selected set of numerical tests to demonstrate the convergence and properties of the proposed scheme. In all tests, the tolerance for PG method is set as $10^{-6}.$

Errors are measured in the following discrete $l_{2}$ norm:  
$$
err=\left (\sum_{1\leq j \leq N} h| u_j^{n}-U_{j}^{n}|^2\right )^{1/2}.
$$
Here $u_{j}^n$ and $U_j^n$ denotes the numerical solutions and  reference solutions at $(x_j,t_n)$.  In what follows we take $u^n_j=\rho_{i,j}^n$,or $\phi_{j}^n$ at time $t=n\tau.$
\subsection{1D multiple species }
We apply our scheme to solve the 1D two-species PNP system (\ref{PNP1}) and verify the proven properties. 
\begin{example}\label{Ex31} (Accuracy test)
We consider the following PNP system 
\begin{equation}\label{Exacc}
\begin{aligned}
  \partial_t \rho_1  =& \partial_x \left( \partial_x \rho_1+ \rho_1 \partial_x \phi \right) ,  \\
  \partial_t \rho_2  =& \partial_x \left( \partial_x \rho_2- \rho_2 \partial_x \phi \right) ,  \\
  -\partial_x^2 \phi  =&\rho_1-\rho_2, 
\end{aligned}
\end{equation}
in $[-1,\ 1]$ and $t>0$.  This is (\ref{PNP1}) with $D_1=D_2=\epsilon=1$, $q_1=1$, $q_2=-1$, and $f(x)=0.$ The initial and boundary conditions are chosen as
\begin{equation}\label{Exacc_BC}
\begin{aligned}
 & \rho_1^{in}(x)=2-x^2, \quad \rho_2^{in}(x)=2+sin(\pi x), \\
& \phi(0,t)=-1, \ \phi(1,t)=1.
\end{aligned}
\end{equation}
In the accuracy test, we consider the numerical solutions obtained by $h=1/320$ and $\tau=1/10000$ as the reference solution. We use the time step $\tau=h$ and $\tau=h^2 $ to compute numerical solutions. The errors and orders at $t=0.5$ are listed in Table 1 and Table 2.
 
  \begin{table}[ht]\label{ex311}
        \centering
                \caption{\tiny{Accuracy for Example \ref{Ex31} with $\tau=h$ }}
        \begin{tabular}{|c| c |c |c| c |c|c| c |}
            \hline
              h& \ $\rho_1$ error & order& $\rho_2$ error &order &$\phi$ error &order \\ [0.5ex] 
                          \hline
          1/10 &   2.67958E-02&-&  9.80117E-03&- &1.17890E-03&-\\
            1/20 &  1.27689E-02&  1.06937 &   4.12484E-03&  1.24862  &5.46161E-04&1.11004  \\
            1/40&   6.20098E-03&  1.04207 &   1.91422E-03& 1.10758   &3.18396E-04& 0.77850 \\
            1/80 &   3.04165E-03&  1.02764&  9.21957E-04&  1.05399  &1.74525E-04& 0.86739  
             \\ [1ex]
            \hline
        \end{tabular}
     \end{table}

  \begin{table}[ht]\label{ex312}
        \centering
                \caption{\tiny{Accuracy for Example \ref{Ex31} with $\tau=h^2$ }}
        \begin{tabular}{|c| c |c |c| c |c|c| c |}
            \hline
              h& \ $\rho_1$ error & order& $\rho_2$ error &order &$\phi$ error &order \\ [0.5ex] 
                          \hline
          1/10 &   9.00817E-03&-&  3.13854E-03&- &1.42932E-03&-\\
            1/20 &  2.22285E-03&  2.01882 &   7.47122E-04&  2.07068 &3.62387E-04&1.97973  \\
            1/40&   5.36781E-04&  2.05001 &   1.78121E-04& 2.06849   & 9.15204E-05& 1.98537 \\
            1/80 &   1.15944E-04&  2.21091&   3.79348E-05&  2.23126  &2.35563E-05& 1.95798  
             \\ [1ex]
            \hline
        \end{tabular}
     \end{table}
     We see from Table 1 and Table 2 that the scheme is first order in time and second order in space.
\end{example}
\begin{example} In this test, still with the initial boundary value problem (\ref{Exacc})-(\ref{Exacc_BC}), we show the proven solution properties. We take $h=0.05, \tau=0.01$ to compute the numerical solutions up to $T=2$. Solutions at $T=0,\ 0.05,\ 0.25,\ 1.5,\ 2$ are given in Figure 1. In Figure 2 are total mass of $\rho_1$ and $\rho_2$, and free energy profile. We see from Figure 1 and Figure 2 that the scheme is positivity preserving, mass conservative, and energy dissipating.

\begin{figure}[h]
\caption{Solution evolutions for $\rho_1, \rho_2$, and $\phi$. }
\centering  
\subfigure{\includegraphics[width=0.48\linewidth]{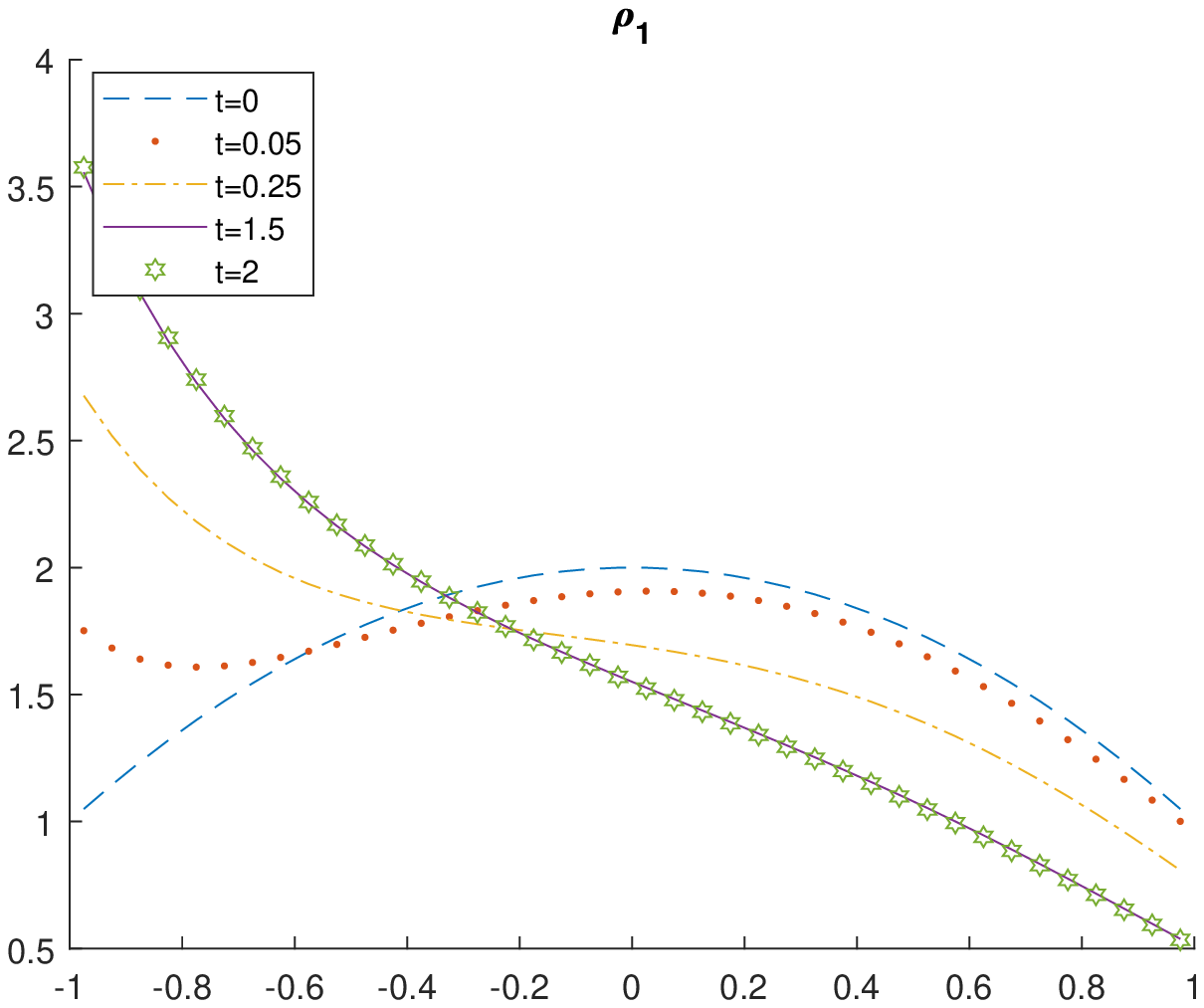}}
\subfigure{\includegraphics[width=0.48\linewidth]{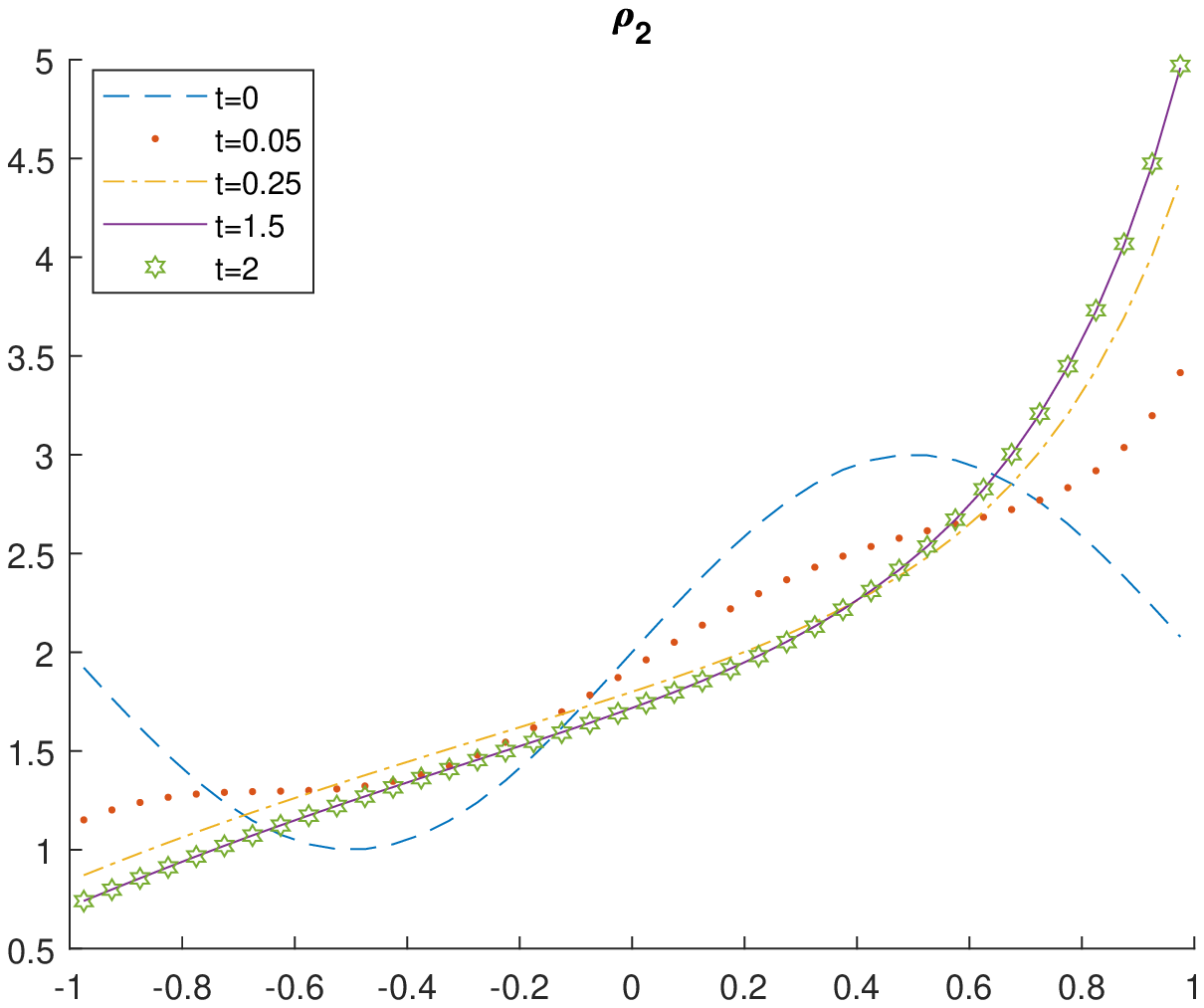}}
\subfigure{\includegraphics[width=0.48\linewidth]{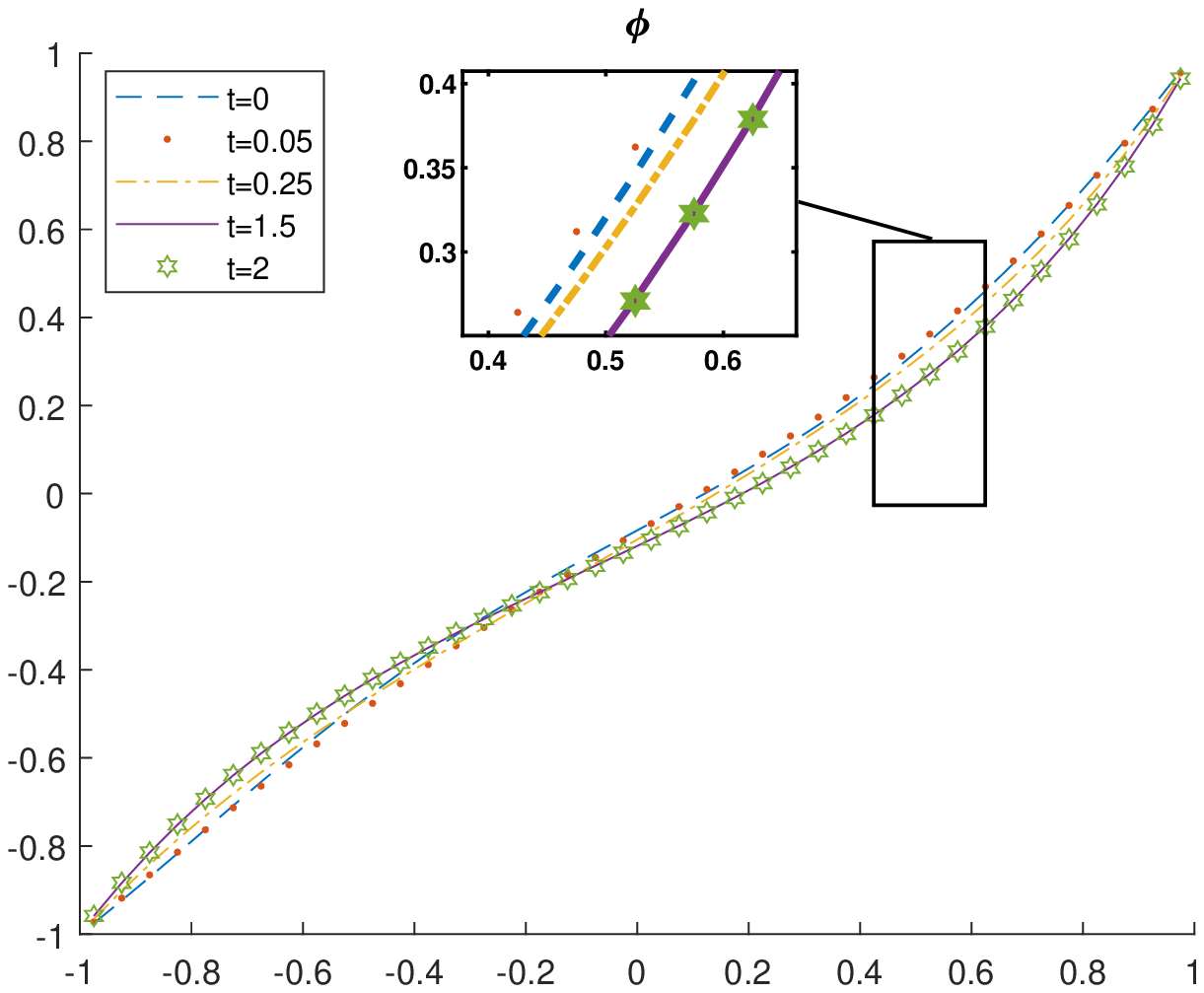}}
\end{figure}
 \begin{figure}[!tbp]
\caption{Energy dissipation and mass conservation }
\centering  
\subfigure{\includegraphics[width=0.48\linewidth]{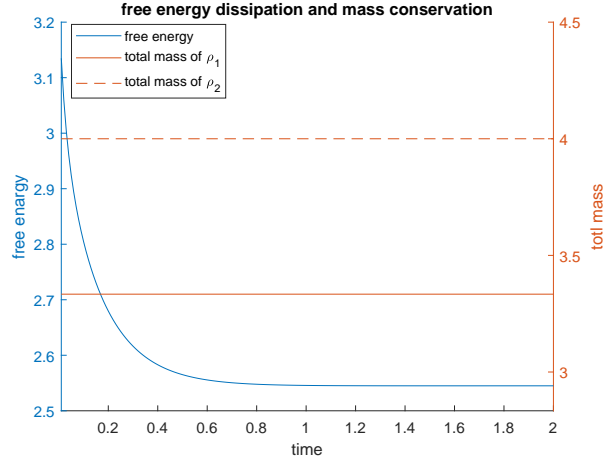}}
\end{figure}
\end{example}

\begin{example} (Positivity propagation)
In this test, we consider the PNP system (\ref{Exacc}) with following  initial and boundary conditions
\begin{equation}\label{EX3BC}
\begin{aligned}
 & \rho_1^{in}(x)=\frac{10}{3}\chi_{_{[-0.5, 0.5]}} , \quad \rho_2^{in}(x)=2+sin(\pi x), \\
& \phi(0,t)=-1, \ \phi(1,t)=1.
\end{aligned}
\end{equation}
We take $h=0.05, \tau=0.01$ to compute the numerical solutions up to $T=2$. Solutions at $T=0,\ 0.015,\ 0.1,\ 1,\ 2$ are displayed in Figure 3. In Figure 4 are total mass of $\rho_1$,  $\rho_2$, and free energy profile. From these results we see that the scheme is positivity preserving, mass conservative, and energy dissipating. We also observe that steady state solutions are identical to those in Example 5.2; this suggests that steady state solutions of the PNP systems with Dirichlet boundary condition only depends on the total mass and the Dirichlet boundary condition, but not sensitive to the profile of the initial data.

In Table 3 we compare CPU times (in seconds) for the PDIP method and the PG method. Here we set $T=0.5$ and choose different number of sub-intervals.
 \begin{table}[ht]\label{comparison}
        \centering
                \caption{\tiny{CPU times comparison for PDIP method and PG method  }}
        \begin{tabular}{|c| c |c |c| c |c|c| c |c|}
            \hline
              h&  1/10  &1/50  & 1/100  &1/150 & 1/200& 1/20 & 1/300 \\ [0.5ex] 
                          \hline
          PDIP &   1.22&2.37& 8.38&17.73&31.32&48.97& 74.71\\
            PG &  0.19& 0.39& 1.15 &   2.02& 3.12  &4.56&6.38  \\
            [1ex]
            \hline
        \end{tabular}
     \end{table}

\begin{figure}[h]
\caption{Solution evolutions for $\rho_1, \rho_2$. }
\centering  
\subfigure{\includegraphics[width=0.48\linewidth]{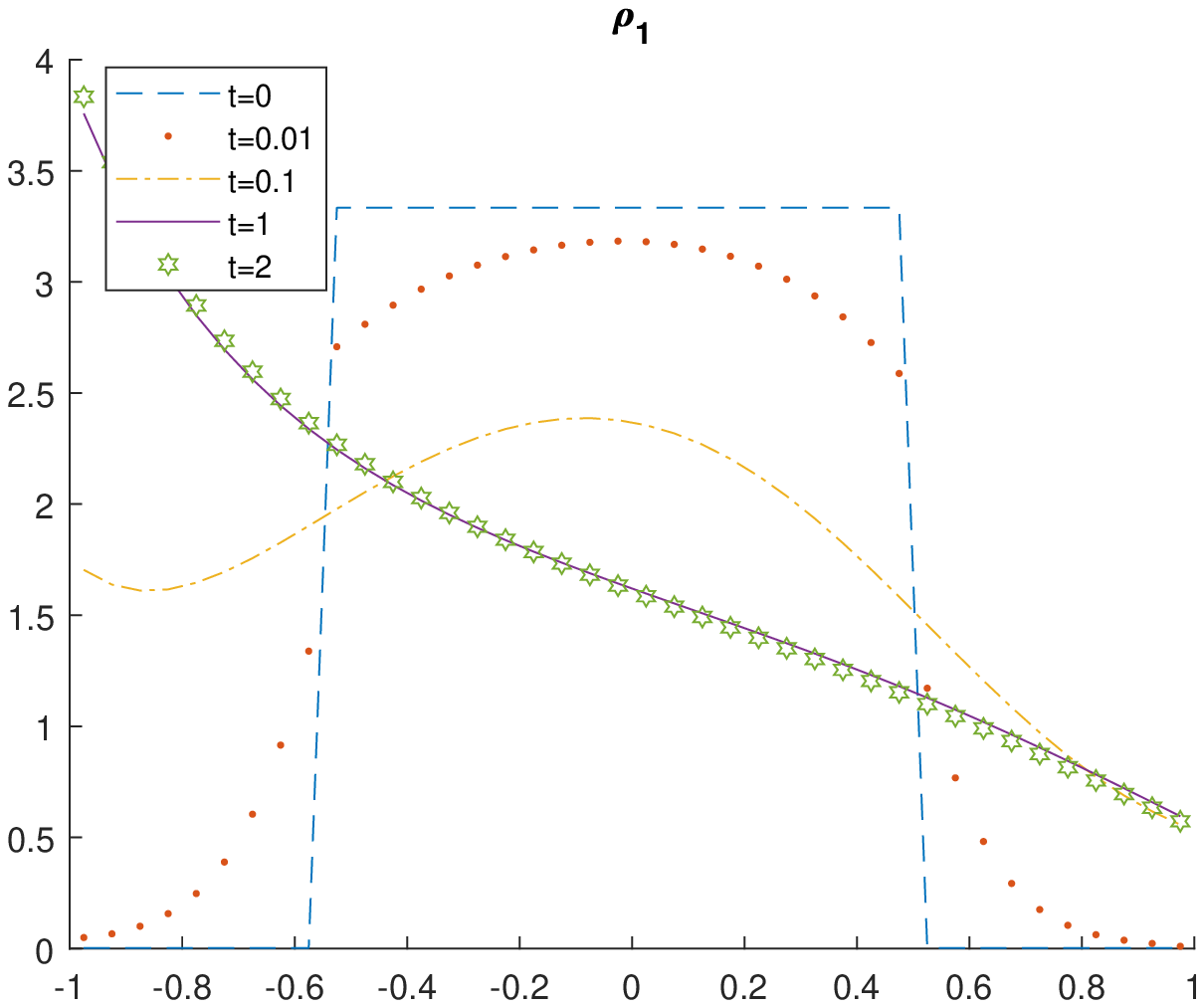}}
\subfigure{\includegraphics[width=0.48\linewidth]{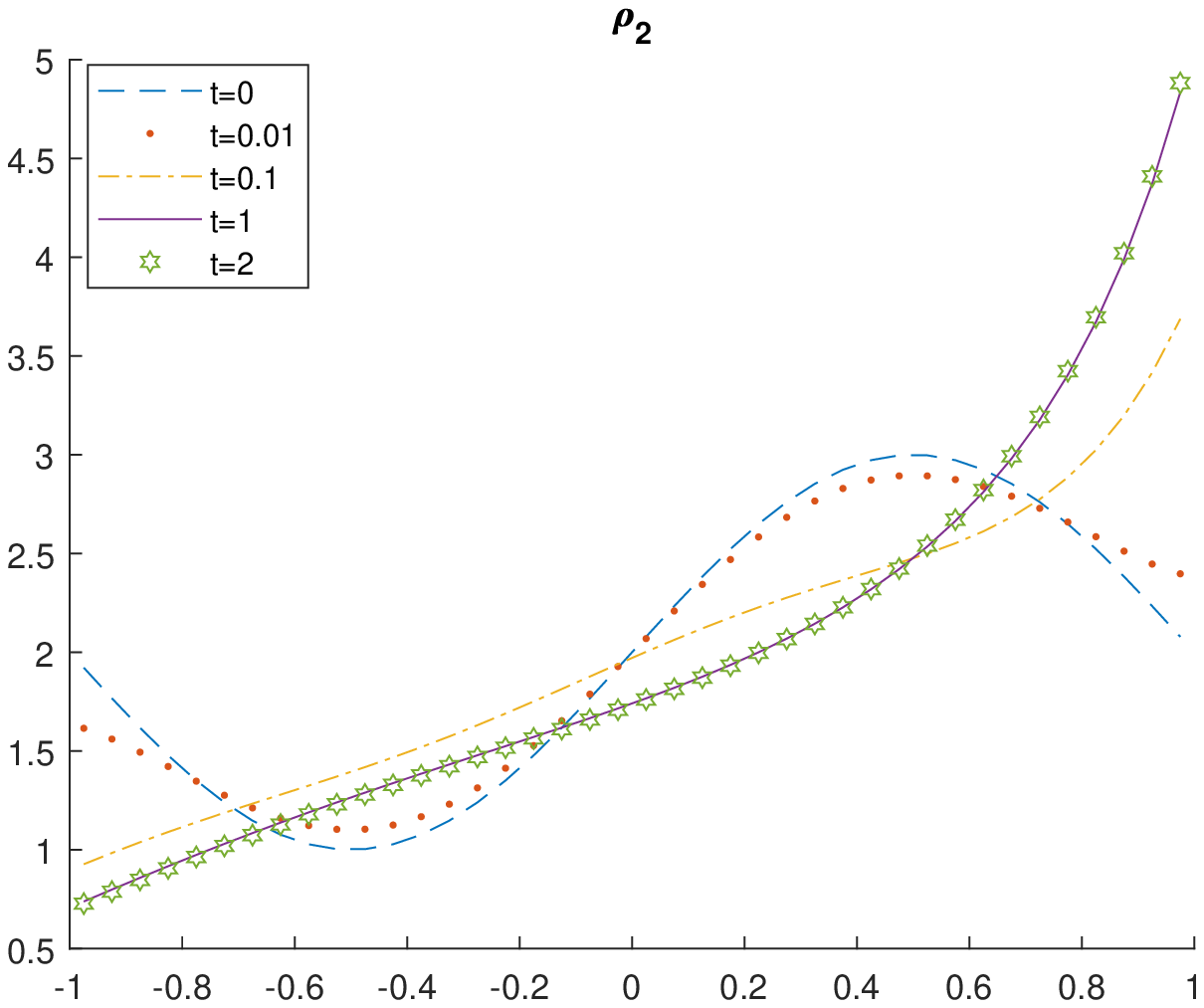}}
\end{figure}
 \begin{figure}[!tbp]
\caption{Energy dissipation and mass conservation }
\centering  
\subfigure{\includegraphics[width=0.48\linewidth]{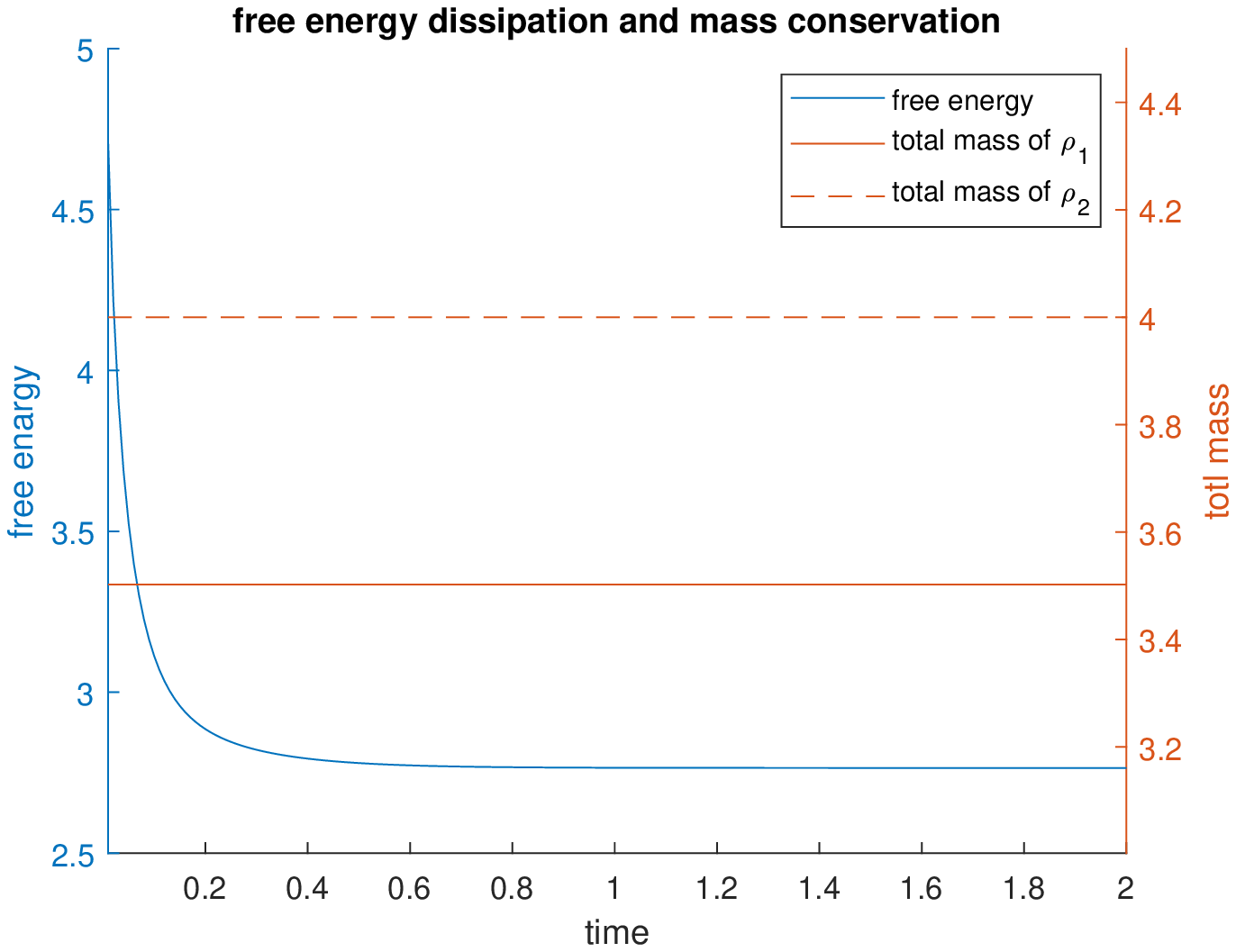}}
\end{figure}
\end{example}

\begin{example} 2D single species (Neumann boundary condition). We now apply our scheme to solve the 2D single-species PNP system
\begin{align*}
  \partial_t \rho  =& \nabla \cdot \left( \nabla \rho+ \rho \nabla \phi \right) ,  \\
  -\Delta  \phi  =& \rho+f(x,y),
\end{align*}
on domain $\Omega=[0,1]\times [0,1]$. We consider the initial boundary conditions 
$$
\rho^{in}(x,y)=-4(x^2-x)-8(y^2-y),\quad \frac{\partial \phi}{\partial n}|_{\partial \Omega}=-1.
$$
The permanent charge $f(x,y)$ is
\begin{equation}\label{BC1}
   f(x,y)= \begin{dcases}
        32, &  \frac{5}{8}\leq x \leq \frac{7}{8} ,\quad \frac{5}{8}\leq y\leq \frac{7}{8}, \\
       0, &   else.  \\
    \end{dcases}
\end{equation}

This problem satisfies the compatibility condition (\ref{CC}). We take $h_x=h_y=0.025, \tau=0.01$ to compute the numerical solutions up to $T=6$. Color plot of the solutions at $T=0.01,0.5,1,2,4,6$ are given in Figure 5. In Figure 6 are total mass of $\rho$ and free energy profile

 \begin{figure}[h]
\caption{Solution evolutions for $\rho$. }
\centering  
\subfigure{\includegraphics[width=0.32\linewidth]{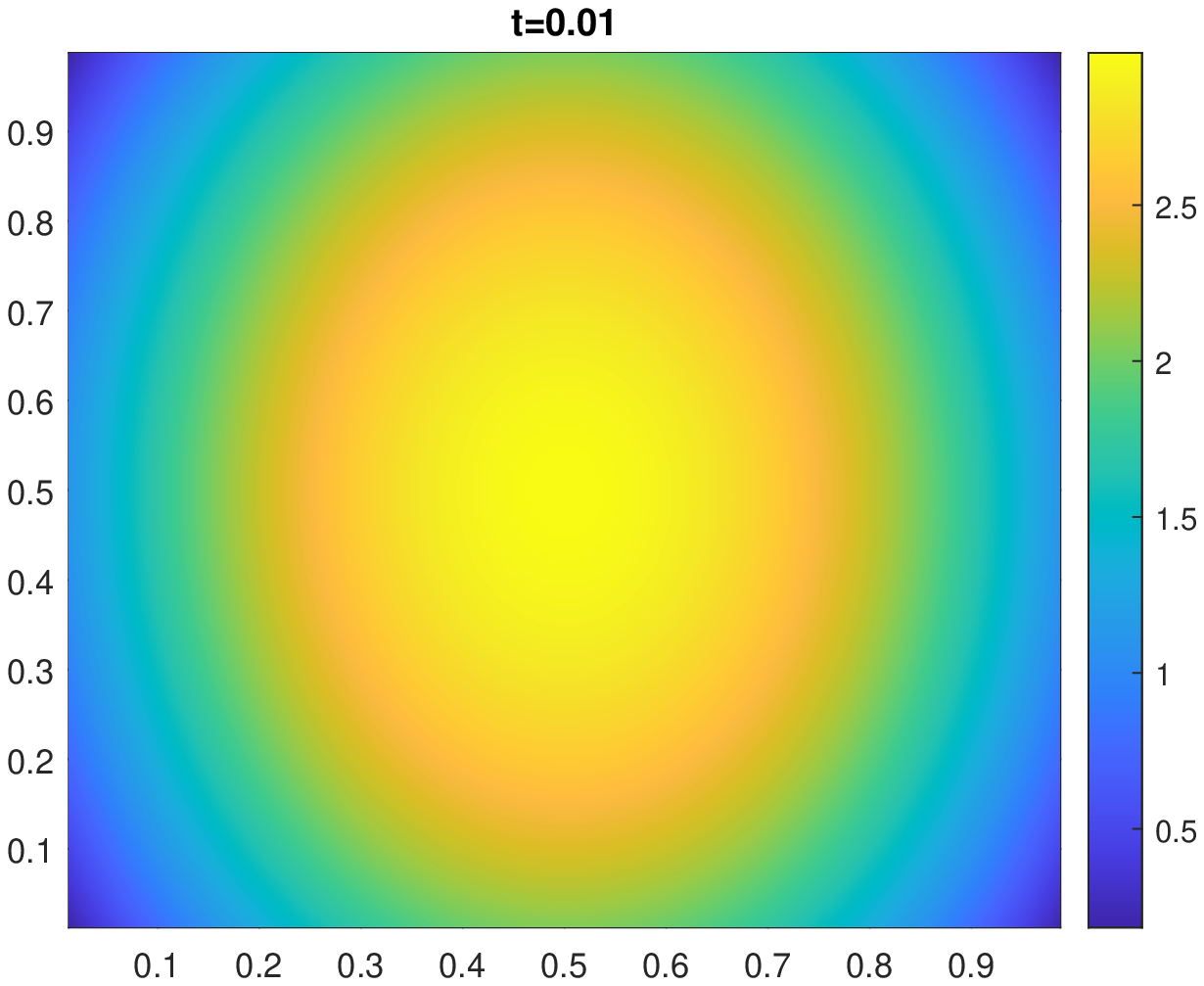}}
\subfigure{\includegraphics[width=0.32\linewidth]{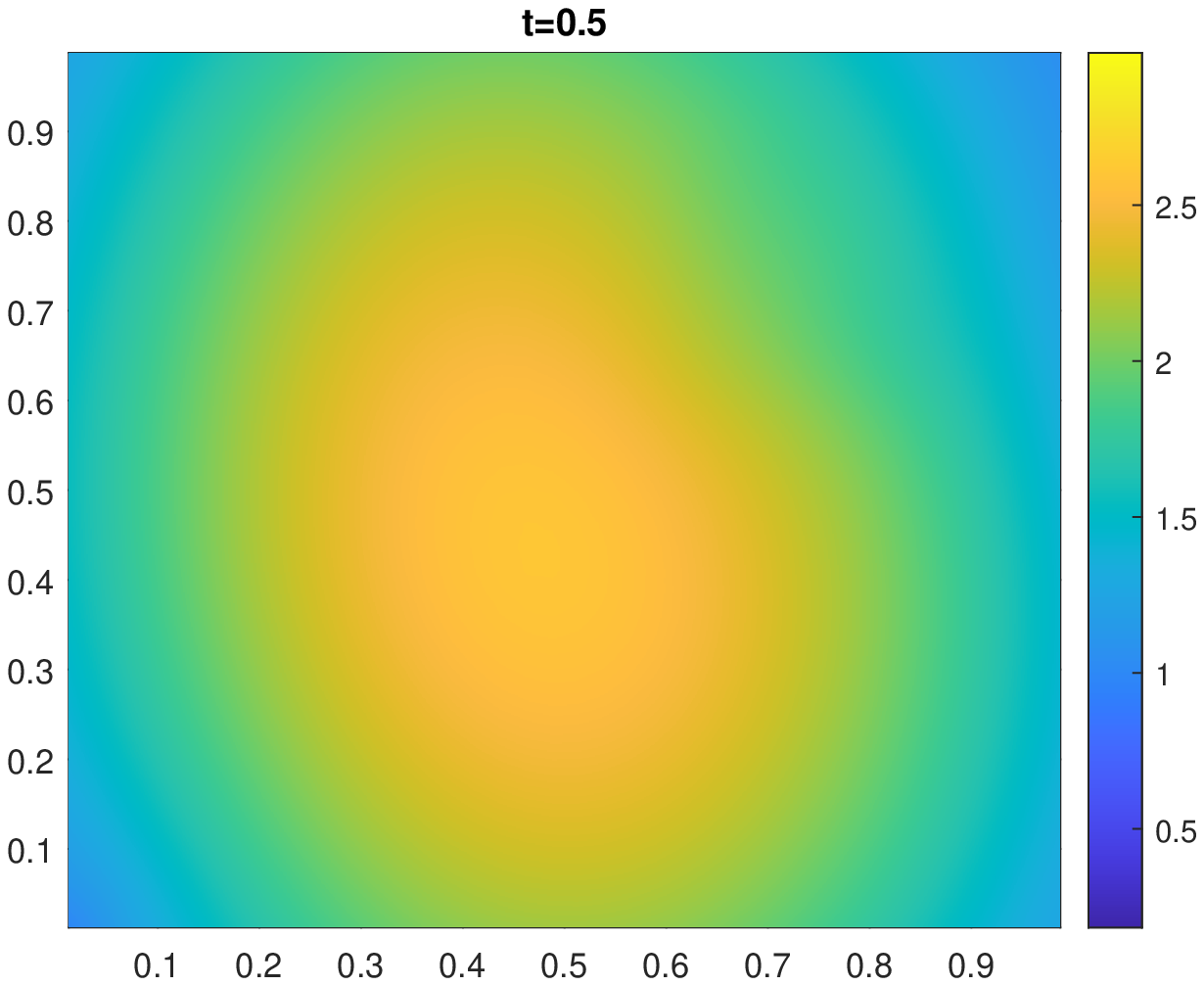}}
\subfigure{\includegraphics[width=0.32\linewidth]{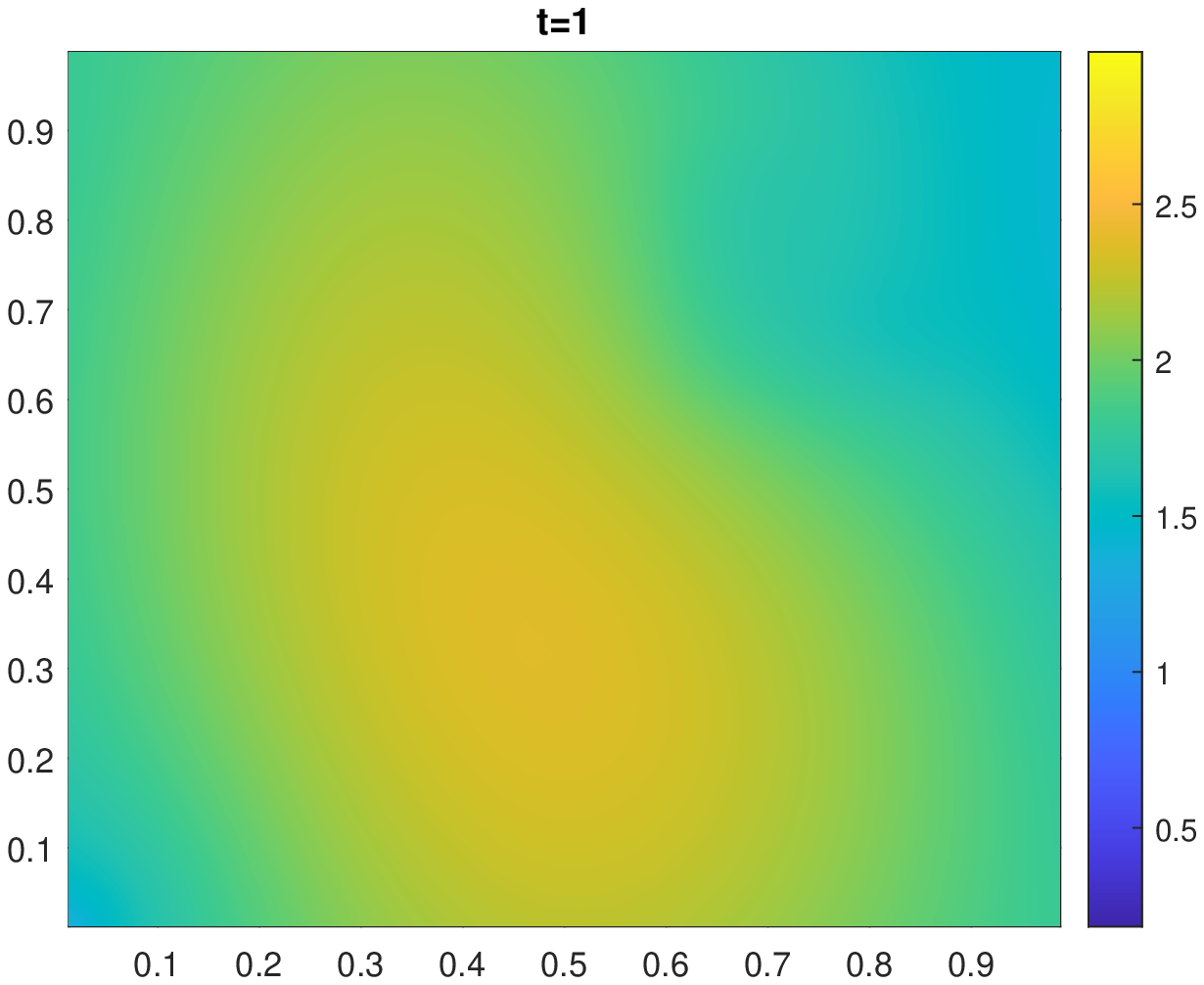}}
\subfigure{\includegraphics[width=0.32\linewidth]{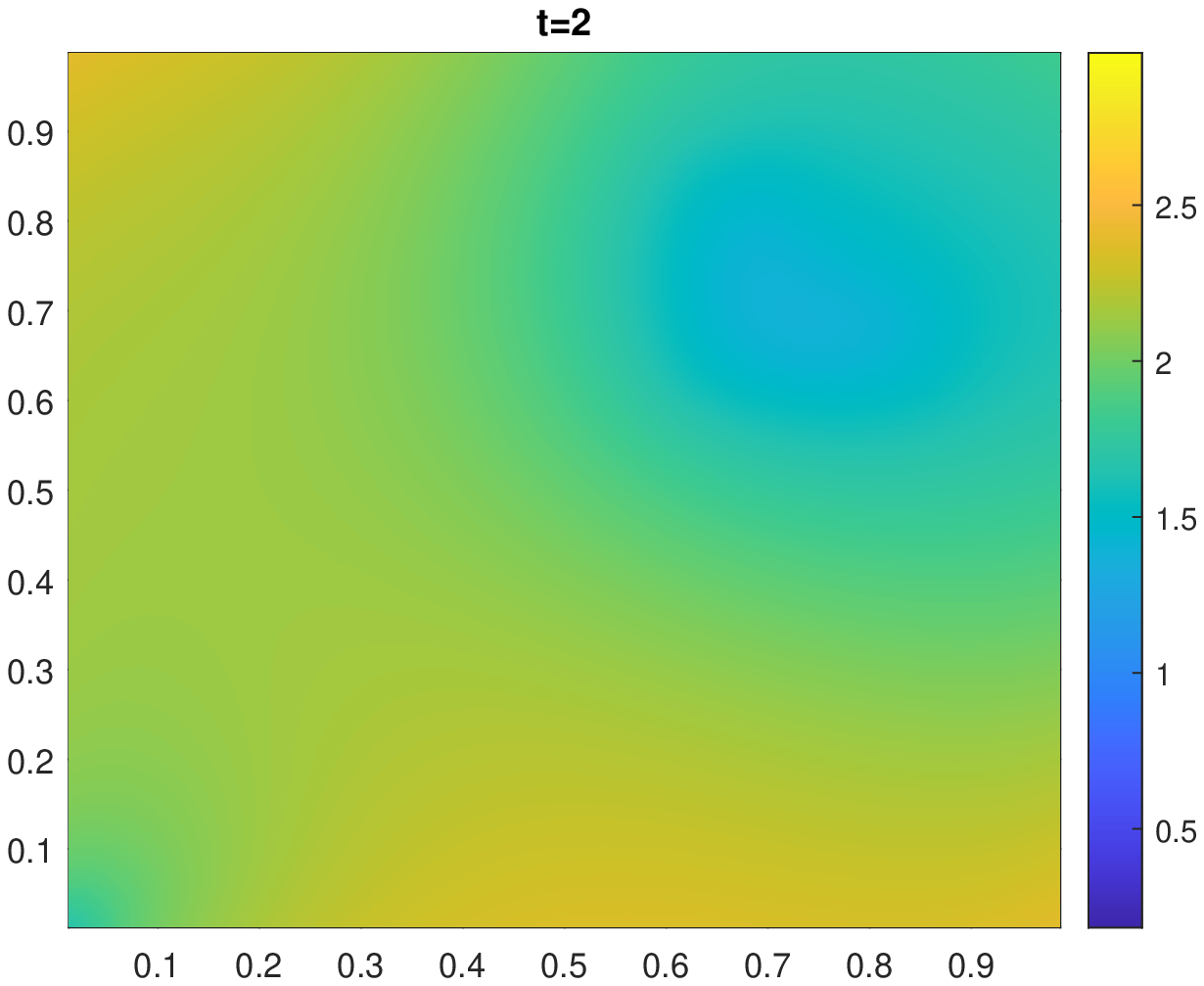}}
\subfigure{\includegraphics[width=0.32\linewidth]{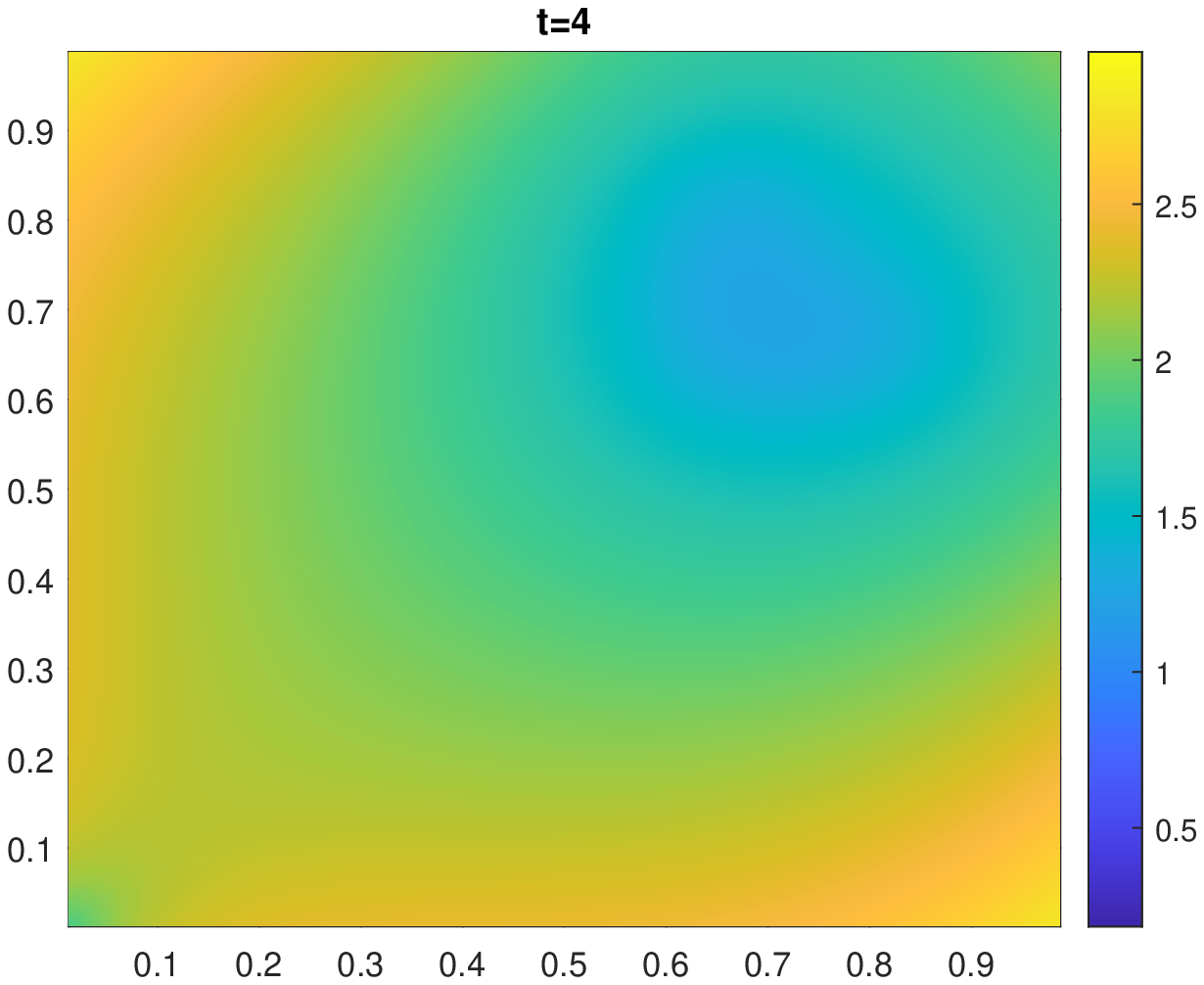}}
\subfigure{\includegraphics[width=0.32\linewidth]{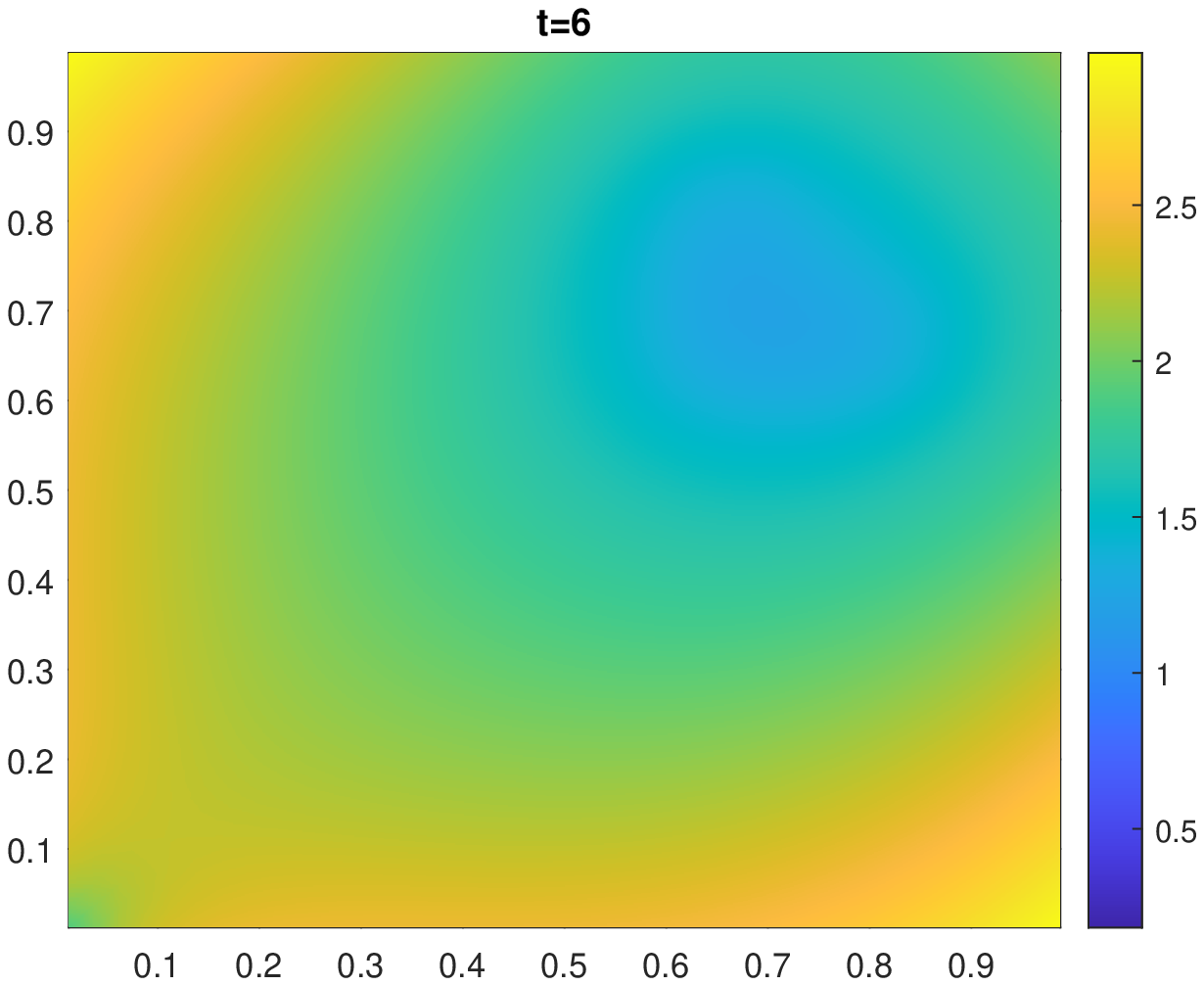}}
\end{figure}

 \begin{figure}[!tbp]
\caption{Energy dissipation and mass conservation }
\centering  
\subfigure{\includegraphics[width=0.48\linewidth]{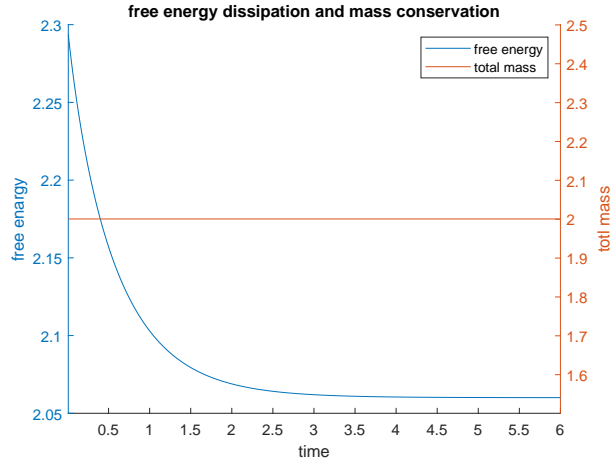}}
\end{figure}
\end{example}

\section{Concluding remarks}

In this paper, a dynamic mass transport method for the PNP system is established by drawing ideas from  both the JKO-type scheme \cite{JKO98, KMX17} and the classical Bennamou-Breiner formulation \cite{BB00}. The energy estimate resembles the physical energy law that governs the PNP system in the continuous case, where the JKO type formulation is an essential component for preserving intrinsic solution properties. Both mass conservation and the energy stability are shown to hold, irrespective of the size of time steps. To reduce computational cost, we use a local approximation for the artificial time in the constraint transport equation by a one step difference and the integral in time by a one term quadrature. 

Furthermore, by imposing a centered finite difference discretization in spatial variables, we establish the solvability of the constrained optimization problem. This also leads to a remarkable result: for any fixed time step and spatial meth size, density positivity will be propagating over all time steps, which is desired for any discrete version of the PNP system.

In the previous section, some numerical experiments were carried out to demonstrate the proven properties of a computed solution. The first experiment numerically verified that the variational scheme yields convergence to the solution of the nonlinear PDE with desired accuracy. 
Secondly, with further examples the computed solutions are also shown to satisfy the energy law for the PNP system, mass conservation, and positivity propagation. It is a matter of future work to prove an error estimate for these numerical solutions. This is not a standard error analysis due to the nonlinearities in these problems, as well as the reformulation as a constrained optimization problem. 

\section*{Acknowledgments}
This research was partially supported by the National Science Foundation under Grant DMS1812666.


\end{document}